\documentclass[12pt,reqno]{article}
\usepackage{amsmath,amssymb,amsthm}
\usepackage{ascmac}
\usepackage{color}
\theoremstyle{plain}
\newtheorem{theorem}{Theorem}[section]
\newtheorem{proposition}[theorem]{Proposition}
\numberwithin{equation}{section}
\newtheorem{corollary}[theorem]{Corollary}
\newtheorem{Remark}[theorem]{Remark}
\newtheorem{Problem}[theorem]{Problem}

\newtheorem{lemma}[theorem]{Lemma}

\def\cH{\mathcal{H}}
\def\bN{\mathbb{N}}
\def\OM{\mathrm{OM}}
\def\OC{\mathrm{OC}}
\def\OMD{\mathrm{OMD}}
\def\eps{\varepsilon}
\def\bR{\mathbb{R}}
\def\bC{\mathbb{C}}
\def\cK{\mathcal{K}}
\def\<{\langle}
\def\>{\rangle}
\def\PMD{\mathrm{PMD}}
\def\PMI{\mathrm{PMI}}

\date{}
\begin{document}

\centerline{\LARGE Ando-Hiai type inequalities for}
\centerline{\LARGE operator means and operator perspectives}

\bigskip
\bigskip
\centerline{\large
Fumio Hiai\footnote{{\it E-mail:} hiai.fumio@gmail.com},
Yuki Seo\footnote{{\it E-mail:} yukis@cc.osaka-kyoiku.ac.jp}
and Shuhei Wada\footnote{{\it E-mail:} wada@j.kisarazu.ac.jp}}

\medskip
\begin{center}
$^1$\,Graduate School of Information Sciences, Tohoku University, \\
Aoba-ku, Sendai 980-8579, Japan
\end{center}

\begin{center}
$^2$\,Department of Mathematics Education, Osaka Kyoiku University, \\
Asahigaoka, Kashiwara, Osaka 582-8582, Japan
\end{center}

\begin{center}
$^3$\,Department of Information and Computer Engineering, \\
National Institute of Technology (KOSEN), \\
Kisarazu College, Kisarazu, Chiba 292-0041, Japan
\end{center}

\begin{abstract}
We improve the existing Ando-Hiai inequalities for operator means and present new ones for
operator perspectives in several ways. We also provide the operator perspective version of
the Lie-Trotter formula and consider the extension problem of operator perspectives to
non-invertible positive operators.

\bigskip\noindent
{\it 2010 Mathematics Subject Classification:}
47A64, 47A63, 47B65

\medskip\noindent
{\it Key words and phrases:}
Operator mean, Operator perspective, Ando-Hiai inequality, Operator monotone function,
Operator convex function, Weak log-majorization, Lie-Trotter formula, Log-Euclidean mean
\end{abstract}


\section{Introduction}

Since the first appearance in the case of weighted operator geometric means in \cite{AH1},
Ando-Hiai type inequalities for operator means have been in active consideration, e.g.,
\cite{Hi3,KS,NS,Se1,Se2,wada1,wada2}, and have taken an important part in recent
developments of multivariable operator means, in particular, of multivariable geometric
means, e.g., \cite{FS1,FSY,HSW,LY,Ya1,Ya2}. When $\sigma$ is a (two-variable) operator
mean (\cite{KA}) and $A,B>0$ are positive invertible operators, the Ando-Hiai inequality is
typically stated as follows:
\begin{align}
&A\sigma B\le I\ \implies\ A^p\sigma B^p\le I,\quad p\ge1, \label{F-1.1}\\
&A\sigma B\ge I\ \implies\ A^p\sigma B^p\ge I,\quad p\ge1. \label{F-1.2}
\end{align}
These have sometimes the slightly stronger formulations as
\begin{align}
&A^p\sigma B^p\le\|A\sigma B\|_\infty^{p-1}(A\sigma B),\qquad\ p\ge1,
\label{F-1.3}\\
&A^p\sigma B^p\ge\lambda_{\min}^{p-1}(A\sigma B)(A\sigma B),\qquad p\ge1,
\label{F-1.4}
\end{align}
where $\|X\|_\infty$ and $\lambda_{\min}(X)$ are the operator norm and the minimum of the
spectrum of a positive invertible operator $X$, respectively.

Among others, a major result in the subject is the characterization of operator means $\sigma$
for which \eqref{F-1.1} or \eqref{F-1.2} holds true, which was given in \cite{wada1} and
says that \eqref{F-1.1} (resp., \eqref{F-1.2}) holds for all $A,B>0$ and $p\ge1$ if and only
if the operator monotone function $f_\sigma$ on $(0,\infty)$ representing $\sigma$ is pmi
(resp., pmd). Here, a positive continuous function $f$ on $(0,\infty)$ is said to be pmi
(power monotone increasing) if $f(t^p)\ge f(t)^p$ for all $t>0$ and $p\ge1$, and pmd (power
monotone decreasing) if the inequality is opposite. Moreover, it was implicitly shown in
\cite{wada1} that the stronger inequalities \eqref{F-1.3} (resp., \eqref{F-1.4}) holds when
$f_\sigma$ is pmi (resp., pmd).

Operator perspectives recently discussed in, e.g., \cite{Ef,ENG,EH} are two-variable
operator functions defined for continuous functions $f$ on $(0,\infty)$ by
$$
P_f(A,B):=B^{1/2}f(B^{-1/2}AB^{-1/2})B^{1/2},\qquad A,B>0.
$$
When $f$ is a positive operator monotone function with $f(1)=1$, the operator perspective
$P_f$ reduces to the operator mean $\sigma_f$ with the representing function $f$ (\cite{KA});
to be precise, $P_f(A,B)=B\sigma_fA$. On the other hand, the operator perspectives for power
functions $f(t)=t^\alpha$ for $\alpha\in\bR\setminus[0,1]$ were formerly treated as
complements of the weighted operator geometric means by several authors (see, e.g.,
\cite{FS2,Fu}). The operator perspectives $P_f$ for operator convex functions have joint
operator convexity (\cite{Ef,ENG}) and are of significant use in quantum information
(\cite{HM}).

The Ando-Hiai inequality has recently been proved in \cite{KS}, together with its stronger
form of log-majorization, for the operator perspectives $P_f$ for power functions
$f(t)=t^\alpha$ with $-1\le\alpha\le0$ (also referred to as matrix geometric means of
negative powers), which implies the inequality for $P_f$ when $f(t)=t^\alpha$,
$1\le\alpha\le2$, as well. Similar result is also contained in \cite{Hi3} for the operator
perspectives $P_f$ when $f(t)=t^\alpha$, $\alpha\ge2$. Motivated by these results, in
the present paper, we consider Ando-Hiai type inequalities for operator perspectives $P_f$
when the functions $f$ on $(0,\infty)$ are more general. Apart from the most typical case
of operator monotone functions $h$, our target functions are operator monotone decreasing
functions $g$, operator convex functions $f$ with $f(0^+)=0$, and functions of the form
$t^nh(t)$ with positive integers $n$ and operator monotone functions $h$. For the operator
perspectives for those functions, we present various Ando-Hiai type inequalities of the
forms \eqref{F-1.1}--\eqref{F-1.4} when $p\ge1$ and their complementary versions when
$0<p\le1$.

The paper is organized as follows. Section 2 is a preliminary, showing close
relations between the above mentioned three kinds of functions -- operator monotone $h$,
operator monotone decreasing $g$, and operator convex $f$ with $f(0^+)=0$. The
characteristics of functions $t^nh(t)$ with operator monotone $h$ are also clarified.

Sections 3 and 4 are main parts of the paper. In Section 3.1 we improve the known Ando-Hiai
inequalities \eqref{F-1.1}--\eqref{F-1.4} for operator means $\sigma_h$ to generalized
stronger forms, together with their complementary versions for $0<p\le1$. Section 3.2
presents new Ando-Hiai type inequalities for the perspectives $P_g$ and $P_f$ when
$g$ and $f$ as  such functions as mentioned above. The typical statements corresponding to
\eqref{F-1.1} and \eqref{F-1.2} are as follows: 
\begin{align}
\mbox{if $f$ is pmi},\,\quad
P_f(A,B) \le I\ &\implies\ P_f(A^p,B^p)\le I,\quad0<p\le1, \label{F-1.5}\\
\mbox{if $f$ is pmd},\quad
P_f(A,B) \ge I\ &\implies\ P_f(A^p,B^p)\ge I,\quad0<p\le1. \label{F-1.6}
\end{align}
when $f$ is an operator convex function with $f(0^+)=0$; the same hold when $g$ is an
operator monotone decreasing function.  Interestingly, the roles of the two parameter regions
$p\ge1$ and $0<p\le1$ are reversed between Sections 3.1 and 3.2. In Section 3.3 some
inequalities in Sections 3.2 are slightly strengthened into weak log-majorizations in the
case of positive definite matrices. Section 3.4 contains an estimation of bounds
which repeatedly appear in the inequalities in Sections 3.1--3.3. In Section 3.5 the range
of parameter $p$ for which the statements in \eqref{F-1.5} and \eqref{F-1.6} hold is
determined, similarly to \cite{HSW,wada2} where the range of $p$ in \eqref{F-1.1}
and \eqref{F-1.2} was determined. In Section 4
we extend the statements \eqref{F-1.5} and \eqref{F-1.6} to the perspectives $P_{t^nh}$ for
the functions $t^nh(t)$ mentioned above when $0<p\le1/2$. But it is left unsettled whether
the statements still hold for the remaining $1/2<p\le1$ or not.

Section 5.1 gives an operator perspective version of the Lie-Trotter formula.
Section 5.2 treats miscellaneous operator norm inequalities for operator means and operator
perspectives related to the Ando-Hiai inequality, including the extension of the results in
\cite{An1,Ya1}. Finally, in Section 6 we consider the extension of operator perspectives to
non-invertible positive operators and extend some inequalities in Sections 3.3, 3.4 and 5.2
to non-invertible case. The existence of such limits as
$\lim_{\eps\searrow0}P_f(A+\eps I,B+\eps I)$ for operator perspectives is quite a non-trivial
problem, while the existence of such limits for operator means is incorporated in their
definition.

\section{Certain positive functions on $(0,\infty)$ and operator perspectives}

 Throughout the paper, $\cH$ is a Hilbert space, $B(\cH)^+$ is the set of bounded positive
operators on $\cH$, and $B(\cH)^{++}$ is the set of invertible $A\in B(\cH)^+$. We also
write $A\ge0$ when $A\in B(\cH)^+$, and $A>0$ when $A\in B(\cH)^{++}$. 

A real continuous function $f$ on $(0,\infty)$ is said to be \emph{operator monotone} if
$$
0<A\le B\ \implies\ f(A)\le f(B)
$$
(where $\cH$ may be any infinite-dimensional Hilbert space), and \emph{operator monotone
decreasing} if $-f$ is operator monotone. Also, $f$ is said to be \emph{operator convex} if
$$
f(\lambda A+(1-\lambda)B)\le\lambda f(A)+(1-\lambda)f(B),\qquad A,B>0,\ \lambda\in[0,1].
$$
For the convenience of presentation, we use the brief notations for the following three
classes of positive functions on $(0,\infty)$:
\begin{align*}
\OM_+&:=\{h:\mbox{operator monotone on}\ (0,\infty),\ h>0\}, \\
\OC_+&:=\{f:\mbox{operator convex on}\ (0,\infty),\ f>0\}, \\
\OMD_+&:=\{g:\mbox{operator monotone decreasing on}\ (0,\infty),\ g>0\}.
\end{align*}
Moreover, we write $\OM_+^1$ for the set of $h\in\OM_+$ with $h(1)=1$, and similarly
$\OC_+^1$ and $\OMD_+^1$.

For any real continuous function $f>0$ on $(0,\infty)$ define its \emph{transpose} function
$\widetilde f$ and its \emph{adjoint} function $f^*$ by
$$
\widetilde f(t):=tf(t^{-1})\quad \mbox{and} \quad f^*(t):=f(t^{-1})^{-1},\qquad t>0.
$$
We set
$$
f(0^+):=\lim_{t\to 0^+}f(t) \quad \mbox{and} \quad f'(\infty):=\lim_{t\to\infty}{f(t)\over t},
$$
whenever these limits exist in $[0,\infty]$. In fact, the limits exist if $f$ is convex or
concave on $(0,\infty)$. If $f$ is a differentiable convex or concave function
on $(0,\infty)$, then $f'(\infty)=\lim_{t\to\infty}f'(t)$, which justifies the notation
$f'(\infty)$. It is easy to verify that $f$ is convex (resp., concave) on $(0,\infty)$ if and
only if so is $\widetilde f$, and moreover
\begin{align}\label{f(0^+)-f'(infty)}
\widetilde f(0^+)=f'(\infty),\qquad{\widetilde f}\,'(\infty)=f(0^+).
\end{align}

The perspective of a real continuous function $f$ on $(0,\infty)$ is a two-variable function
defined by $P_f(x,y):=yf(x/y)$ for $x,y\in(0,\infty)$. The \emph{operator perspective}
associated with $f$ is the extension of $P_f(x,y)$ to operators in $B(\cH)^{++}$ as follows:
\begin{align}\label{F-2.4}
P_f(A,B):=B^{1/2}f(B^{-1/2}AB^{-1/2})B^{1/2},\qquad A,B\in B(\cH)^{++}.
\end{align}
In particular, when $h\in\OM_+$, the operator perspective $P_h(A,B)$ for $A,B>0$ is nothing
but the \emph{operator connection} $B\sigma_hA$ in Kubo-Ando's sense \cite{KA}
corresponding to $h$. Thus, the operator perspectives include the operator connections
(in particular, operator means when $h(1)=1$) as their special case.

For any continuous function $f>0$ on $(0,\infty)$ the following equalities are
easy to verify (as shown in \cite[Lemma 2.1]{HM} for the first): for every $A,B>0$,
\begin{align}
P_{\widetilde f}(A,B)&=P_f(B,A), \label{F-2.5}\\
P_{f^*}(A,B)&=P_f(A^{-1},B^{-1})^{-1}. \label{F-2.6}
\end{align}

Our main aim of the paper is to obtain Ando-Hiai type inequalities for the operator
perspectives $P_f(A,B)$ for the positive functions $f$ on $(0,\infty)$ of the form
$f(t)=t^nh(t)$, where $n\in\bN$ and $h\in\OM_+$. In this section we give several descriptions
of the positive functions on $(0,\infty)$ of such form $t^nh(t)$. Those descriptions may
independently be of some interest, while they are not fully necessary in our later discussions.

The next proposition is concerned with the functions of the form $th(t)$ with $h\in\OM_+$.
The equivalence relations in the proposition are mostly known, while we briefly give the
proof for completeness.

\begin{proposition}\label{P-2.1}
For any function $f>0$ on $(0,\infty)$ set $g:=\widetilde f$ and $h(t):=f(t)/t$
for $t>0$. Then $g(t)=h(t^{-1})$ and the following conditions are equivalent:
\begin{itemize}
\item[\rm(i)] $h\in\OM_+$;
\item[\rm(ii)] $g\in\OMD_+$;
\item[\rm(iii)] $g\in\OC_+$ and $\lim_{t\to\infty}g(t)<\infty$;
\item[\rm(iv)] $f\in\OC_+$ and $f(0^+)=0$;
\item[\rm(v)] $f\in\OC_+$ and $\lim_{t\to0^+}f(t)/t<\infty$.
\end{itemize}
\end{proposition}

\begin{proof}
That $g(t)=h(t^{-1})$ is easily verified, and so (i)\,$\iff$\,(ii) is obvious.
(v)\,$\implies$\,(iv) is also clear. Both (i)\,$\implies$\,(v) and (iv)\,$\implies$\,(i)
are immediately seen from \cite[Theorem 2.4]{HP}. Hence (i), (ii), (iv) and (v) are
equivalent.

For a convex function $g>0$ on $(0,\infty)$, it is obvious that
$\lim_{t\to\infty}g(t)<\infty$ if and only if $g$ is non-increasing. Hence
(ii)\,$\iff$\,(iii) follows from \cite[Theorem 3.1]{AH2}, but we here include a more direct
proof of (iii)\,$\iff$\,(iv). It was shown in \cite[Proposition A.1]{HM} that a real function
$f$ on $(0,\infty)$ is operator convex if and only if so is $\widetilde f$. When $f>0$ is
convex on $(0,\infty)$, we further note that $f(0^+)=0$ $\iff$ ${\widetilde f}\,'(\infty)=0$
$\iff$ $\lim_{t\to\infty}\widetilde f(t)<\infty$. Hence, from $g=\widetilde f$,
(iii)\,$\iff$\,(iv) follows. (Since
$\lim_{t\to0^+}f(t)/t=\lim_{t\to\infty}\widetilde f(t)$, we have (iv)\,$\iff$\,(v) as well.)
\end{proof}

Proposition \ref{P-2.2} says that the classes $\OM_+$, $\OC_+$ and $\OMD_+$ are closely
related to one another. Since $h\in\OM_+$ $\iff$ $\widetilde h\in\OM_+$ $\iff$
$h^*\in\OM_+$ (see \cite{KA}), we see that the class $\{f\in\OC_+:f(0^+)=0\}$ is closed
under the operations corresponding to $h\mapsto h^*$ and $h\mapsto\widetilde h$. When
$h$, $g$ and $f$ are given as above, we have $th^*(t)=f^*(t)$ and
$t\widetilde h(t)=t^2g(t)=t^2\widetilde f(t)$. Hence $\{f\in\OC_+:f(0^+)=0\}$ is closed
under the operations $f\mapsto f^*$ and $f\mapsto t^2\widetilde f(t)$. Furthermore, we note
that
\begin{align}\label{F-2.7}
\{f\in\OC_+:f(0^+)=0\}=\{th(t):h\in\OM_+\}=\{t^2g(t):g\in\OMD_+\}.
\end{align}

The functions in Proposition \ref{P-2.1} can be characterized by properties of their
operator perspectives. For instance, we state the following based on \cite{Ef,ENG}.

\begin{proposition}\label{P-2.2}
Let $f$, $g$ and $h$ be given as in Proposition \ref{P-2.1}. Then the equivalent conditions of
Proposition \ref{P-2.1} are also equivalent to any of the following:
\begin{itemize}
\item[\rm(vi)] $f(0^+)=0$ and $P_f$ is jointly operator convex, i.e.,
$$
P_f(\lambda A_1+(1-\lambda)A_2,\lambda B_1+(1-\lambda)B_2)
\le\lambda P_f(A_1,B_1)+(1-\lambda)P_f(A_2,B_2)
$$
for all $A_i,B_i\in B(\cH)^{++}$ ($i=1,2$) and $\lambda\in[0,1]$;
\item[\rm(vii)] $P_f$ is right operator decreasing, i.e.,
$$
0<B_1\le B_2\ \implies\ P_f(A,B_1)\ge P_f(A,B_2)
$$
for any (equivalently, some) $A>0$;
\item[\rm(viii)] $P_g$ is left operator decreasing, i.e.,
$$
0<A_1\le A_2\ \implies\ P_g(A_1,B)\ge P_g(A_2,B)
$$
for any (equivalently, some) $B>0$.
\end{itemize}
\end{proposition}

\begin{proof}
(iv)\,$\iff$\,(vi) is \cite[Theorem 2.2]{ENG}. (ii)\,$\iff$\,(viii) is immediately seen
since $g(A)=B^{-1/2}P_g(B^{1/2}AB^{1/2},B)B^{-1/2}$. (vii)\,$\iff$\,(viii) is
obvious from \eqref{F-2.5}.
\end{proof}

To characterize the functions of the form $t^nh(t)$ with $n\ge2$ and $h\in\OM_+$, we need
the notion of operator $k$-tone functions. The original definition of $k$-tone functions in
\cite{FHR} is not so simple, so we here give, among many others, its two equivalent
conditions, restricted to real functions on $(0,\infty)$, see
\cite[Definition 1.4, Theorems 3.3 and 5.1]{FHR} for more details. A real function
$f$ on $(0,\infty)$ is \emph{operator $k$-tone} if and only if any of the following
conditions holds:
\begin{itemize}
\item[(A)] $f$ is $C^{k-2}$ on $(0,\infty)$ (this is void for $k=1$) and
$f^{[k-1]}(x,\alpha,\dots,\alpha)$ with $k-1$ $\alpha$'s is operator monotone on $(0,\infty)$
for some (equivalently, any) $\alpha\in(0,\infty)$ (with continuation of value at $x=\alpha$),
where $f^{[k-1]}$ is the $(k-1)$st divided difference of $f$;
\item[(B)] $f$ is analytic on $(0,\infty)$ and
$$
{d^k\over dt^k}\,f(A+tX)\Big|_{t=0}\ge0
$$
for every $A\in B(\cH)^{++}$ and $X\in B(\cH)^+$, where $\cH$ is infinite-dimensional (the
above derivative of order $k$ can be defined in the operator norm).
\end{itemize}

In particular, condition (A) reduces L\"owner's characterization of operator monotone
functions \cite{Lo} when $k=1$, and to Kraus' characterization of operator convex functions
\cite{Kr} when $k=2$; a concise exposition on L\"owner's and Kraus' theories is found in
\cite[Section 2.4]{Hi2}. Thus, the $1$-tonicity and the $2$-tonicity are nothing but the
operator monotonicity and the operator convexity, respectively.

The next proposition is the characterization of the functions $t^nh$ with $h\in\OM_+$. When
$n=1$, conditions (a), (c) and (d) are (i), (iv) and (v) of Proposition \ref{P-2.1},
respectively, and (b) is incorporated in the equalities in \eqref{F-2.7}. Since we shall not
directly use this proposition in the subsequent sections, the reader may skip its proof that
heavily depends on \cite{FHR}.

\begin{proposition}\label{P-2.3}
For any function $f>0$ on $(0,\infty)$ and $n\in\bN$, the following conditions are
equivalent:
\begin{itemize}
\item[$(${\rm a}$)$] $f(t)=t^nh(t)$, $t>0$, with $h\in\OM_+$;
\item[$(${\rm b}$)$] $f(t)=t^{n+1}g(t)$, $t>0$, with $g\in\OMD_+$;
\item[$(${\rm c}$)$] $f$ is operator $(n+1)$-tone on $(0,\infty)$ and
$\lim_{t\to0^+}f(t)/t^{n-1}=0$;
\item[$(${\rm d}$)$] $f$ is operator $(n+1)$-tone on $(0,\infty)$ and
$\lim_{t\to0^+}f(t)/t^n<\infty$.
\end{itemize}
\end{proposition}

\begin{proof}
(a)\,$\iff$\,(b).\enspace
For functions $h>0$ and $g(t):=h(t)/t$ on $(0,\infty)$, note that
$h\in\OM_+$ $\iff$ $g\in\OMD_+$. Hence (a)\,$\iff$\,(b) follows.

(a)\,$\implies$\,(d).\enspace
Assume that $f=t^nh$ as stated in (a). For any $\eps>0$ define $f_\eps(t):=(t-\eps)^nh(t)$ for
$t>0$. By \cite[Corollary 3.4]{FHR}, $f_\eps$ is operator $(n+1)$-tone on $(0,\infty)$.
Since $f_\eps(t)\to f(t)$ as $\eps\searrow0$ for $t>0$, $f$ is operator $(n+1)$-tone on
$(0,\infty)$ by \cite[Proposition 3.9]{FHR}. Moreover, since $h>0$ on $(0,\infty)$
from the assumption $f>0$, $\lim_{t\to0^+}f(t)/t^n=\lim_{t\to0^+}h(t)<\infty$.

(a)\,$\implies$\,(c).\enspace
The proof is similar to that of (a)\,$\implies$\,(d) above. For the last part,
$\lim_{t\to0^+}f(t)/t^{n-1}=\lim_{t\to0^+}th(t)=0$.

(c)\,$\implies$\,(a).\enspace
Prove this implication by induction on $n$. Since the operator $2$-tonicity means the
operator convexity, the case $n=1$ holds by (iv)\,$\implies$\,(i) of Proposition \ref{P-2.1}.
Suppose that (c)\,$\implies$\,(a) when $n=m$, and prove the case $n=m+1$. Now, assume (c)
for $n=m+1$. Since $f$ is operator $(m+1)$-tone on $(0,\infty)$, $f$ is analytic in
$(0,\infty)$ by \cite[Lemma 3.1]{FHR} (also by condition (B) above). Let
$\widehat f(t):=f(t)/t$ for $t>0$. Then
$\lim_{t\searrow0}\widehat f(t)/t^{m-1}=\lim_{t\searrow0}f(t)/t^m=0$. For any $\eps>0$,
define
$$
\widehat f_\eps(t):=f^{[1]}(t,\eps)
=\begin{cases}{f(t)-f(\eps)\over t-\eps} & \text{for $t>0$, $t\ne\eps$}, \\
f'(\eps) & \text{for $t=\eps$}.\end{cases}
$$
Then since $f(0^+)=0$, $\widehat f_\eps(t)\to\widehat f(t)$ as $\eps\searrow0$ for all $t>0$.
Furthermore, it is easy to see that
$$
\widehat f_\eps^{[m]}(t,\underbrace{\eps,\dots,\eps}_m)
=f^{[m+1]}(t,\underbrace{\eps,\dots,\eps}_{m+1}),\qquad t>0,
$$
where $\widehat f_\eps^{[m]}$ is the $m$th divided difference of $\widehat f_\eps$. By using
\cite[Theorem 3.3]{FHR} twice, it follows that $\widehat f_\eps$ is operator $m$-tone
on $(0,\infty)$. Hence $\widehat f$ is operator $m$-tone by
\cite[Proposition 3.9]{FHR}. By the induction hypothesis for $n=m$,
$\widehat f(t)=t^{m-1}h(t)$, $t>0$, with $h\in\OM_+$, so that $f(t)=t^mh(t)$. Hence
(c)\,$\implies$\,(a) when $n=m+1$ is proved.

(d)\,$\implies$\,(a).\enspace
The proof is similar to that of (c)\,$\implies$\,(a) with slight modifications, where the
initial case $n=1$ of induction on $n$ is (v)\,$\implies$\,(i) of Proposition \ref{P-2.1}.
\end{proof}

\section{Ando-Hiai type inequalities}

When $f$ is a continuous function on $(0,\infty)$ such that $f>0$ and $f(1)=1$, we consider,
for a positive real number $p$, the following statements for the operator perspective $P_f$:
\begin{align}
A,B>0,\ P_f(A,B) \le I\ &\implies\ P_f(A^p,B^p)\le I, \label{ando-hiai}\\
A,B>0,\ P_f(A,B) \ge I\ &\implies\ P_f(A^p,B^p)\ge I. \label{ando-hiai2}
\end{align}
These statements were first shown in \cite{AH1} in the case where $f(t)=t^\alpha$ with
$0\le\alpha\le1$ so that $P_f(A,B)=B\#_\alpha A:=B^{1/2}(B^{-1/2}AB^{-1/2})^{\alpha}B^{1/2}$,
the \emph{weighted operator geometric mean}. So we refer to \eqref{ando-hiai} and
\eqref{ando-hiai2} as \emph{Ando-Hiai} (or \emph{AH} for short) type inequalities.
The correspondences $P_f\leftrightarrow P_{\widetilde f}$ and
$P_f\leftrightarrow P_{f^*}$ based on \eqref{F-2.5} and \eqref{F-2.6} will be useful for
our discussions on AH type inequalities. In particular, note that $P_f$ satisfies
\eqref{ando-hiai} if and only if $P_{f^*}$ satisfies \eqref{ando-hiai2}.

In the case where $f\in\OM_+^1$, we have the following basic result about statements
\eqref{ando-hiai} and \eqref{ando-hiai2}. As noted in Section 2, $P_f(A,B)=B\sigma_fA$ for
$A,B>0$.
 
\begin{proposition}[\cite{wada1}] \label{P-3.1}
Assume that $f\in\OM_+^1$. Then the following conditions are equivalent:
\begin{itemize}
\item[$(${\rm i}$)$] $P_f$ (or $\sigma_f$) satisfies \eqref{ando-hiai2} for all $p\ge1$;
\item[$(${\rm ii}$)$] $f(t)^p\le f(t^p)$ for all $t>0$, $p\ge1$.
\end{itemize}
\end{proposition}

We say that $f$ is \emph{power monotone increasing} (\emph{pmi} for short) if it satisfies
condition (ii) of Proposition \ref{P-3.1}. On the other hand, $f$ is said to be
\emph{power monotone decreasing} (\emph{pmd} for short) if $f^*$ is pmi, i.e.,
$f(t)^p\ge f(t^p)$ for all $t>0$, $p\ge1$. Also, as noted in \cite{wada1}, it is
clear from the correspondence $P_f\leftrightarrow P_{f^*}$ that if $f\in \OM_+^1$, then
$P_f$ satisfies \eqref{ando-hiai} for all $p\ge 1$ if and only if $f$ is pmd.

In this section we shall first refine the known AH inequality for operator means and show
its complementary versions. Then we discuss AH type inequalities for operator perspectives
associated with functions described in Propositions \ref{P-2.1}--\ref{P-2.3},
other than those in $\OM_+^1$.

\subsection{Operator means}

In this subsection we present several AH type inequalities for operator means, which
generalize and supplement the AH inequality stated in Proposition \ref{P-3.1} and further
discussed recently in \cite{HSW} in a more general setting of multivariable operator means.
The next theorem is a generalized version of the AH inequality though restricted
to $1\le p\le2$, together with its complementary version for $0<p\le1$. Our stress
here is that the inequalities hold for general operator means without the pmi or pmd
assumption on their representing functions. For a positive invertible operator $X>0$ let
$\|X\|_\infty$ be the operator norm of $X$ and $\lambda_{\min}(X)$ be the minimum
of the spectrum of $X$.

\begin{theorem}\label{T-3.2}
Let $h\in\OM_+^1$ and $A,B>0$. Set $C:=A^{-1/2}BA^{-1/2}$. Then the following
inequalities hold:
\begin{align}
A^p\sigma_hB^p&\ge\lambda_{\min}\biggl({h(C^p)\over h(C)^p}\biggr)
\lambda_{\min}^{p-1}(A\sigma_hB)(A\sigma_hB)
\quad\mbox{for $1\le p\le2$}, \label{F-7.1}\\
A^p\sigma_hB^p&\le\bigg\|{h(C^p)\over h(C)^p}\bigg\|_\infty
\|A\sigma_hB\|_\infty^{p-1}(A\sigma_hB)
\qquad\ \ \mbox{for $1\le p\le2$}, \label{F-7.2}\\
A^p\sigma_hB^p&\le\bigg\|{h(C^p)\over h(C)^p}\bigg\|_\infty
\lambda_{\min}^{p-1}(A\sigma_hB)(A\sigma_hB)
\qquad\ \,\mbox{for $0<p\le1$}, \label{F-7.3}\\
A^p\sigma_hB^p&\ge\lambda_{\min}\biggl({h(C^p)\over h(C)^p}\biggr)
\|A\sigma_hB\|_\infty^{p-1}(A\sigma_hB)
\quad\ \,\mbox{for $0<p\le1$}. \label{F-7.4}
\end{align}
\end{theorem}

\begin{proof}
When $1\le p\le2$, the proof of \cite[Lemma 2.1]{wada1} shows that
$$
A\sigma_hB\ge I\ \implies\ A^p\sigma_hB^p
\ge\lambda_{\min}\biggl({h(C^p)\over h(C)^p}\biggr)(A\sigma_hB).
$$
Indeed, from the proof in \cite{wada1} we find that if $A\sigma_h B \ge I$, then
\begin{align*}
A^p\sigma_h B^p & \ge A^{1/2}h(C^p)h(C)^{1-p}A^{1/2}\\
& \geq \lambda_{\min}\left( \frac{h(C^p)}{h(C)^p}\right) A^{1/2}h(C)A^{1/2}\\
& = \lambda_{\min}\left( \frac{h(C^p)}{h(C)^p}\right)(A\sigma_h B).
\end{align*}
For every $A,B>0$, apply the above to $\alpha^{-1}A$ and $\alpha^{-1}B$ with
$\alpha:=\lambda_{\min}(A\sigma_h B)$ to show \eqref{F-7.1}. Inequality \eqref{F-7.2}
immediately follows from \eqref{F-7.1} by replacing $h$, $A$ and $B$ in \eqref{F-7.1} with
$h^*$, $A^{-1}$ and $B^{-1}$.

Next, when $0< p\le1$, we show that
$$
A\sigma_hB\ge I\ \implies\ A^p\sigma_hB^p\le 
\bigg\| {{h(C^p)}\over {h(C)^p}}\bigg\|_\infty (A\sigma_h B). 
$$
Assume that $A\sigma_hB\ge I$; then $h(C)\ge A^{-1}$ and the L\"{o}wner-Heinz theorem gives
$h(C)^{1-p}\ge A^{-(1-p)}$ since $0\le 1-p\le 1$. Hence we have
\begin{align*}
A^p\sigma_hB^p& = A^{p/2}h(A^{-p/2}(A^{1/2}CA^{1/2})^pA^{-p/2})A^{p/2} \\
&=A^{p/2}h(A^{1-p\over2}(A^{-1}\#_pC)A^{1-p\over2})A^{p/2} \\
&=A^{1/2}(A^{-(1-p)}\sigma_h(A^{-1}\#_pC))A^{1/2}\\
&\le A^{1/2}(h(C)^{1-p}\sigma_h(h(C) \#_pC))A^{1/2}\\
&= A^{1/2}(h(C)^{1-p}h(C^p))A^{1/2} \\
&\le \bigg\|{{h(C^p)}\over {h(C)^p}}\bigg\|_\infty A^{1/2}h(C) A^{1/2}
= \bigg\|{{h(C^p)}\over{h(C)^p}}\bigg\|_\infty (A\sigma_h B).
\end{align*}
Hence inequality \eqref{F-7.3} is shown as in the above proof of \eqref{F-7.1}, and
\eqref{F-7.4} follows from \eqref{F-7.3} as \eqref{F-7.2} does from \eqref{F-7.1}.
\end{proof}

The general formulation of Theorem \ref{T-3.2} explicitly specifies the role of the
pmi (or pmd) assumption on $h$ in the AH inequality in \cite{wada1},
thus giving the inequalities under the pmi (pmd) assumption as follows:

\begin{corollary}\label{C-3.3}
If $h\in\OM_+^1$ is pmi, then
\begin{align}
A^p\sigma_hB^p&\ge\lambda_{\min}^{p-1}(A\sigma_hB)(A\sigma_hB),
\qquad p\ge1, \label{F-7.5}\\
A^p\sigma_hB^p&\le\lambda_{\min}^{p-1}(A\sigma_hB)(A\sigma_hB),
\qquad 0< p\le1. \label{F-7.6}
\end{align}
If $h\in\OM_+^1$ is pmd, then
\begin{align}
A^p\sigma_hB^p&\le\|A\sigma_hB\|_\infty^{p-1}(A\sigma_hB),
\qquad p\ge1, \label{F-7.7}\\
A^p\sigma_hB^p&\ge\|A\sigma_hB\|_\infty^{p-1}(A\sigma_hB),
\qquad 0< p\le1. \label{F-7.8}
\end{align}
\end{corollary}

\begin{proof}
Note that $h$ is pmi (resp., pmd), then $h(C^p)\ge h(C)^p$ (resp., $h(C^p)\le h(C^p)$) when
$p\ge1$, and the inequalities are reversed when $0<p\le1$. Hence \eqref{F-7.5} and
\eqref{F-7.7} for $1\le p\le2$ as well as \eqref{F-7.6} and \eqref{F-7.8} immediately
follow from \eqref{F-7.1}--\eqref{F-7.4}. Inequalities \eqref{F-7.5} and
\eqref{F-7.7} for general $p\ge1$ can be seen by a simple induction argument as in the
last part of the proof of \cite[Theorem 3.1]{HSW}. We here give the proof of \eqref{F-7.5}
for completeness. Assume that \eqref{F-7.5} is true when $1\le p\le 2^k$, and extend it to
$1\le p\le2^{k+1}$. When $2^k<p\le2^{k+1}$, letting $p=2p'$ with $2^{k-1}<p'\le2^k$ one has
\begin{align*}
A^p\sigma_hB^p&\ge\lambda_{\min}(A^{p'}\sigma_hB^{p'})(A^{p'}\sigma_hB^{p'}) \\
&\ge\lambda_{\min}\bigl(\lambda_{\min}^{p'-1}(A\sigma_hB)(A\sigma_hB)\bigr)
\cdot\lambda_{\min}^{p'-1}(A\sigma_hB)(A\sigma_hB) \\
&=\lambda_{\min}^{2p'-1}(A\sigma_hB)(A\sigma_hB)
=\lambda_{\min}^{p-1}(A\sigma_hB)(A\sigma_hB).
\end{align*}
\end{proof}

The AH inequalities are conventionally written in the forms \eqref{ando-hiai} and
\eqref{ando-hiai2}, whose stronger formulations are \eqref{F-7.5} and \eqref{F-7.7} as
discussed in \cite{HSW}. The inequalities in \eqref{F-7.6} and \eqref{F-7.8}, complementary
respectively to \eqref{F-7.5} and \eqref{F-7.7}, are new, but we note that those
complementary versions do not have conventional forms like \eqref{ando-hiai} and
\eqref{ando-hiai2}.

Although it does not seem possible to extend the inequalities in \eqref{F-7.1} and
\eqref{F-7.2} to $p>2$, we have their modifications which hold for all $p\ge1$.

\begin{proposition}\label{P-3.4}
For every  $h\in \OM_+^1$ and every $A,B>0$, 
$$
\lambda_{\min}\biggl({{h(C_p^p)} \over {h(C_p)^p}}\biggr)
\lambda_{\min}^{p-1}(A\sigma_h B)(A\sigma_h B)
\le A^p \sigma_h B^p
\le\bigg\|{{h(C_p^p)} \over {h(C_p)^p}}\bigg\|_\infty
\|A\sigma_h B\|_\infty^{p-1}(A\sigma_h B)
$$
for all $p\ge1$, where $C_p:=(A^{-p/2}B^p A^{-p/2})^{1/p}$. 
\end{proposition}

\begin{proof}
It follows from \cite[Corollary 4.6]{HSW} that 
\begin{equation}\label{absorbing}
\lambda_{\min}^{p-1}(A\sigma_h B)(A\sigma_h B) 
\le A^p \sigma_{h_{[1/p]}} B^p 
\le
\|A\sigma_h B\|_\infty^{p-1}(A\sigma_h B),
\end{equation}
where $h_{[1/p]}(t):=h(t^{1/p})^p$; here note that $h_{[1/p]}\in\OM_+^1$ again.
The first inequality in \eqref{absorbing} implies that
\begin{align*}
&\lambda_{\min}^{p-1}(A\sigma_h B)
\bigl[( A^p \sigma_h B^p )^{-1/2}(A\sigma_h B)
( A^p \sigma_h B^p )^{-1/2}\bigr] \\
&\qquad\le( A^p \sigma_h B^p )^{-1/2}
\bigl(A^p \sigma_{h_{[1/p]}} B^p\bigr)
( A^p \sigma_h B^p )^{-1/2} \\
&\qquad\le
\big\|( A^p \sigma_h B^p )^{-1/2}
\bigl(A^p \sigma_{h_{[1/p]}} B^p\bigr)
( A^p \sigma_h B^p )^{-1/2}
\big\|_\infty I \\
&\qquad=
r\bigl( (A^p \sigma_{h_{[1/p]}} B^p)
\left( A^p \sigma_h B^p \right)^{-1}\bigr) I \\
&\qquad=
r\biggl( A^{p/2}\,{{h_{[1/p]}(C_p^p)} \over {h(C_p^p)}}\,A^{-p/2}\biggr)I \\
&\qquad=
r\biggl({{h_{[1/p]}(C_p^p)} \over {h(C_p^p)}}\biggr)I
=\bigg\|{{h(C_p)^p} \over {h(C_p^p)}}\bigg\|_\infty I,
\end{align*}
where $r(X)$ denotes the spectral radius of $X$. Hence the first asserted inequality
is obtained. The second inequality is shown in a similar way to the above with use of the
second inequality in \eqref{absorbing}.
\end{proof}

\subsection{Operator perspectives}

The aim of this subsection is to prove AH type inequalities for $g\in\OMD_+^1$ and
$f\in\OC_+^1$ with $f(0^+)=0$. We first note a basic fact about functions
$f$ satisfying \eqref{ando-hiai}.

\begin{proposition}\label{P-3.5}
Let $f>0$ be a continuous function on $(0,\infty)$ and $p>0$. If $P_f$ satisfies
\eqref{ando-hiai} for $p$, then $f(t^p)\le f(t)^p$ for all $t>0$.
\end{proposition}

\begin{proof}
For any $t>0$, since $P_f\left({t\over {f(t)}}, {1\over {f(t)}}\right)=1$, we have 
$$
P_f\left(\left({t\over {f(t)}}\right)^p, \left({1\over {f(t)}}\right)^p\right)\le 1,
$$
which implies that $f(t^p)\le f(t)^p$. 
\end{proof}

\begin{corollary}\label{C-3.6}
Let $f>0$ be a continuous function on $(0,\infty)$.
If $P_f$ satisfies \eqref{ando-hiai} for all $p\in (0,1)$, then $f$ is pmi.
\end{corollary}

Now, we are ready to show the following theorem, which says that the pmi (pmd)
characterization of operator means satisfying the AH inequality can be expanded
to certain relevant operator perspectives.

\begin{theorem}\label{T-3.7}
Let $h\in\OM_+^1$. Then the following conditions are equivalent:
\begin{itemize}
\item[\rm(i)] $h$ is pmi (resp., pmd);
\item[\rm(ii)] $P_h$ satisfies \eqref{ando-hiai2} (resp., \eqref{ando-hiai}) for all
$p\in[1,\infty)$;
\item[\rm(iii)] $P_{1/h}$ satisfies \eqref{ando-hiai2} (resp., \eqref{ando-hiai}) for all
$p\in(0,1]$;
\item[\rm(iv)] $P_{h(1/t)}$ satisfies \eqref{ando-hiai} (resp., \eqref{ando-hiai2}) for all
$p\in(0,1]$;
\item[\rm(v)] $P_{th}$ satisfies \eqref{ando-hiai} (resp., \eqref{ando-hiai2}) for all
$p\in(0,1]$.
\end{itemize}
\end{theorem}

\begin{proof}
Noting the correspondence $P_f\leftrightarrow P_{f^*}$, we may prove only the result
when $h$ is pmi. Set $g(t):=h(1/t)$. (i)\,$\iff$\,(ii) is Proposition \ref{P-3.1}.
(iii)\,$\iff$\,(iv) follows from \eqref{F-2.6} since $1/h=g^*$. (iv)\,$\iff$\,(v) follows
from \eqref{F-2.5} since $\widetilde g(t)=th(t)$.

(iv)\,$\implies$\,(i).\enspace
Assume (iv), i.e., $P_g$ satisfies \eqref{ando-hiai} for all $p\in(0,1)$. Hence Corollary
\ref{C-3.6} implies that $g$ is pmi and so is $h$.

(i)\,$\implies$\,(iv).\enspace
Assume (i). Let $A,B>0$ and assume that $P_g(A,B)\le I$. Put $C:=B^{1/2}A^{-1}B^{1/2}$ so that
$A^{-1}=B^{-1/2}CB^{-1/2}$. Then $P_g(A,B)\le I$ is equivalent to 
$$
h(C)=g(B^{-1/2}AB^{-1/2})\le B^{-1},\quad\mbox{or}\ \ B\le h(C)^{-1}.
$$
Assume that $p\in[1/2,1]$. Note that $B^{1/2}A^{-p}B^{1/2}=B\#_pC$. With the operator mean
$\sigma_h$ corresponding to $h$, we thus have 
\begin{align}
P_g(A^p,B^p)&=B^{p/2}h(B^{p/2}A^{-p}B^{p/2})B^{p/2} \nonumber\\
&=B^p(B^{-p}\sigma_h A^{-p})B^p \nonumber\\
&=B^{p-{1\over2}}(B^{1-p} \sigma_h (B\#_p C))B^{p-{1\over2}} \nonumber\\
&\le B^{p-{1\over2}}(h(C)^{p-1}\sigma_h (h(C)^{-1}\#_p C))B^{p-{1\over2}} \nonumber\\
&=B^{p-{1\over2}}h(C)^{p-1}h(C^p)B^{p-{1\over2}} \nonumber\\
&\le B^{p-{1\over2}}h(C)^{2p-1}B^{p-{1\over2}}, \label{F-3.3}
\end{align}
where the last inequality is derived from the assumption that $h$ is pmi. Since
$h(C)^{2p-1}\le B^{-(2p-1)}$ thanks to $0\le2p-1\le1$, we now obtain
$$
P_g(A,B)\le I\ \implies\ P_g(A^p, B^p) \le I
$$
for all $p\in[1/2,1]$. Iterating this yields (iv).
\end{proof}

By Theorem \ref{T-3.7} with Proposition \ref{P-2.1} we have the following AH type
inequalities for operator perspectives associated with certain functions in $\OMD_+^1$ and
$\OC_+^1$.

\begin{corollary}\label{C-3.8}
If $g\in\OMD_+^1$ is pmi (resp., pmd), then $P_g$ satisfies \eqref{ando-hiai} (resp.,
\eqref{ando-hiai2}) for all $p\in(0,1]$. The same statement holds for $P_f$ when
$f\in\OC_+^1$ with $f(0^+)=0$ is pmi (resp., pmd) in place of $g$.
\end{corollary}

\begin{proof}
Set $h(t):= g(1/t)$ for $t>0$; then $h\in \OM_+^1$ by Proposition \ref{P-2.1}.
The statement for $P_g$ follows from (i)\,$\implies$\,(iv) of
Theorem~\ref{T-3.7}. The statements for $P_f$ immediately follow from those for $P_g$,
where $g:=\widetilde{f}\in \OMD_+^1$, by using Proposition \ref{P-2.1} and \eqref{F-2.5}
(or (i)\,$\implies$\,(v) of Theorem~\ref{T-3.7}).
\end{proof}

The following is a generalized version of the above corollary with no restriction on $g$
and $f$, though restricted to $p\in[1/2,1]$.

\begin{proposition}\label{P-3.9}
Let $g\in\OMD_+^1$ and $A,B>0$. Set $C:=B^{-1/2}AB^{-1/2}$. Then for every $p\in[1/2,1]$,
\begin{align*}
P_g(A,B)&\le I\ \implies\ P_g(A^p,B^p)\le\biggl\|{g(C^p)\over g(C)^p}\bigg\|_\infty I, \\
P_g(A,B)&\ge I\ \implies\ P_g(A^p,B^p)\ge\lambda_{\min}\biggl({g(C^p)\over g(C)^p}\biggr)I,
\end{align*}

The same statements hold for $P_f$ when $f\in\OC_+^1$ with $f(0^+)=0$.
\end{proposition}

\begin{proof}
Assume that $P_g(A,B)\le I$. The inequality in \eqref{F-3.3} yields that
\begin{align}\label{F-3.13}
P_g(A,B)\le I\ \implies\ P_g(A^p,B^p)\le\biggl\|{g(C^p)\over g(C)^p}\bigg\|_\infty
B^{p-{1\over2}}g(C)^{2p-1}B^{p-{1\over2}},
\end{align}
where $h(C)$ with $C=B^{1/2}A^{-1}B^{1/2}$ is replaced here by $g(C)$ with
$C=B^{-1/2}AB^{-1/2}$. Since $g(C)\le B^{-1}$ and $0\le2p-1\le1$, we have
$B^{p-{1\over2}}g(C)^{2p-1}B^{p-{1\over2}}\le I$. Hence the first statement
follows. Then it is immediate to show the second by replacing $g,A,B$ with
$g^*,A^{-1},B^{-1}$.

When $f\in\OC_+^1$ with $f(0^+)=0$, we have $g:=\widetilde f\in\OMD_+^1$ by Proposition
\ref{P-2.1}. Since
\begin{align}\label{F-3.14}
{f(C^p)\over f(C)^p}={C^pg(C^{-p})\over(Cg(C^{-1}))^p}
={g(C^{-p})\over g(C^{-1})^p},
\end{align}
we note that
\begin{align}\label{bound-f-g}
\bigg\|{f(C^p)\over f(C)^p}\bigg\|_\infty
=\bigg\|{g((B^{1/2}A^{-1}B^{1/2})^p)\over g(B^{1/2}A^{-1}B^{1/2})^p}\bigg\|_\infty
=\bigg\|{g((A^{-1/2}BA^{-1/2})^p)\over g(A^{-1/2}BA^{-1/2})^p}\bigg\|_\infty
\end{align}
and similarly for $\lambda_{\min}$. In view of \eqref{F-2.5}, the result for
$P_f$ follows from that of $P_g$ by interchanging $A$ and $B$.
\end{proof}

We remark that the situation for $P_g$ and $P_f$ in Proposition \ref{P-3.9} is not so good as
that for operator means in the previous subsection, since
$B^{p-{1\over2}}g(C)^{2p-1}B^{p-{1\over2}}$ in \eqref{F-3.13} is different from
$P_g(A,B)^{2p-1}$.

We next consider a complementary version of Proposition \ref{P-3.9} for
$p\in[1,2]$. To do this, we need an extra constant of Kantorovich type.
Recall the \emph{generalized Kantorovich constant} $K(\xi,p)$ defined by
\begin{equation} \label{eq:K}
K(\xi,p) := \frac{\xi^p-\xi}{(p-1)(\xi-1)}
\left( \frac{p-1}{p}\cdot \frac{\xi^p-1}{\xi^p-\xi} \right)^p \quad
\mbox{for $\xi>1$ and $p\in\mathbb{R}$},
\end{equation}
where $K(\xi,1):=\lim_{p\to 1}K(\xi,p)=1$, see \cite[Definition 2.2]{FMPS}. It is known in
\cite[Theorem 4.3]{FMPS} that if
$B\le A$ with either $m\le A \le M$ or $m\le B\le M$ for some scalars
$0<m\leq M$, then $B^p\le K(\xi,p)A^p$ for all $p>1$, where $\xi:=M/m$.

\begin{proposition}\label{P-3.10}
Let $f\in\OC_+^1$ with $f(0^+)=0$ and $A,B>0$. Set $C:=A^{1/2}B^{-1}A^{1/2}$ and
$\xi:=\|A\|_\infty/\lambda_{\min}(A)$ (i.e., the condition number of $A$).
For every $p\in[1,2]$,
\begin{align*}
P_{f}(A,B) &\leq I\ \implies
\ P_{f}(A^p,B^p) \leq K(\xi,2p-1)\bigg\|{f(C^p)\over f(C)^p}\bigg\|_\infty
\lambda_{\min}^{1-p}(P_{f}(A,B))I, \\
P_{f}(A,B) &\ge I\ \implies
\ P_{f}(A^p,B^p) \ge K(\xi,2p-1)^{-1}\lambda_{\min}\biggl({f(C^p)\over f(C)^p}\biggr)
\|P_{f}(A,B)\|_\infty^{1-p}I,
\end{align*}
where $K(\xi, 2p-1)$ is the generalized Kantorovich constant in \eqref{eq:K}.\par

The same statements hold for $P_g$ when $g\in\OMD_+^1$ and
$\xi:=\|B\|_\infty/\lambda_{\min}(B)$.
\end{proposition}

\begin{proof}
Set $h(t):=\widetilde{f}(t^{-1})=t^{-1}f(t)$ for $t>0$; then $h\in\OM_+^1$ by
Proposition~\ref{P-2.1} and $P_f(A,B)=P_{\widetilde f}(B,A)=A^{1/2}h(C)A^{1/2}$, where
$C:=A^{1/2}B^{-1}A^{1/2}$ and so $B=A^{1/2}C^{-1}A^{1/2}$. Assume that $P_f(A,B)\le I$,
i.e., $h(C)\le A^{-1}$. For any $p\in[1,2]$, since $0\le2-p\le1$,
\begin{align*}
P_f(A^p,B^p)&=P_{\widetilde f}(B^p,A^p)
=A^{p/2}h(A^{p/2}(A^{1/2}C^{-1}A^{1/2})^{-p}A^{p/2})A^{p/2} \\
&=A^{p/2}h\bigl(A^{p-1\over2}CA^{-1/2}(A^{1/2}C^{-1}A^{1/2})^{2-p}
A^{-1/2}CA^{p-1\over2}\bigr)A^{p/2} \\
&=A^{p/2}h\bigl(A^{p-1\over2}C(A^{-1}\#_{2-p}C^{-1})CA^{p-1\over2}\bigr)A^{p/2} \\
&=A^{p-{1\over2}}\bigl(A^{1-p}\sigma_h[(CA^{-1}C)\#_{2-p}C]\bigr)A^{p-{1\over2}}.
\end{align*}
Now, set $\lambda:=\|A^{-1/2}h(C)^{-1}A^{-1/2}\|_\infty=\lambda_{\min}^{-1}(P_f(A,B))$. Since
$0\le p-1\le1$ and $A^{-1}\le\|h(C)^{-1/2}A^{-1}h(C)^{-1/2}\|_\infty h(C)=\lambda h(C)$,
we have $A^{1-p}\le(\lambda h(C))^{p-1}$ and $CA^{-1}C\le\lambda C^2h(C)$, which imply that
\begin{align}
P_f(A^p,B^p)&\le A^{p-{1\over2}}\bigl((\lambda h(C))^{p-1}\sigma_h
[(\lambda C^2h(C))\#_{2-p}C]\bigr)A^{p-{1\over2}} \nonumber\\
&=\lambda^{p-1}A^{p-{1\over2}}h(C)^{p-1}h(C^p)A^{p-{1\over2}} \nonumber\\
&=\lambda^{p-1}A^{p-{1\over2}}\bigl(f(C)^{-p}f(C^p)h(C)^{2p-1}\bigr)A^{p-{1\over2}}
\nonumber\\
&\le\lambda^{p-1}\bigg\|{f(C^p)\over f(C)^p}\bigg\|_\infty
A^{p-{1\over2}}h(C)^{2p-1}A^{p-{1\over2}}. \label{F-3.16}
\end{align}
Since $\|A\|_\infty^{-1}\le A^{-1}\le\lambda_{\min}(A)^{-1}$, applying the
Kantorovich inequality mentioned above to $h(C)\le A^{-1}$, we have
$h(C)^{2p-1}\le K(\xi,2p-1)A^{1-2p}$. Therefore,
$$
P_f(A^p,B^p)\le\lambda^{p-1}\bigg\|{f(C^p)\over f(C)^p}\bigg\|_\infty K(\xi,2p-1)I,
$$
which is the inequality in the first assertion.

The proof of the second assertion is similar to the above, so we omit the details.
The statements for $P_g$ immediately follow from those for $P_f$ by using \eqref{F-2.5} and
the arguments in \eqref{F-3.14} and \eqref{bound-f-g}.
\end{proof}

Note that the bounds $\big\|{f(C^p)\over f(C)^p}\big\|_\infty$ and
$\lambda_{\min}\bigl({f(C^p)\over f(C)^p}\bigr)$ (also those for $g$) in Proposition
\ref{P-3.10} are unchanged when $C=A^{1/2}B^{-1}A^{1/2}$ is replaced with
$C=B^{-1/2}AB^{-1/2}$, as in \eqref{bound-f-g}.

We remark that 
\[
\lambda_{\min}^{1-p}(P_{f}(A,B))\bigg\|{f(C^p)\over f(C)^p}\bigg\|_\infty K(\xi,2p-1)=1
\]
in the case of $p=1$.

\begin{corollary} \label{C-3.11}
Let $f\in\OC_+^1$ with $f(0^+)=0$ and $A,B>0$. Set
$\xi:=\|A\|_\infty/\lambda_{\min}(A)$. If $f$ is pmd, then
$$
P_f(A,B) \leq I\ \implies \ P_f(A^p,B^p) \leq K(\xi,2p-1)\lambda_{\min}^{1-p}(P_f(A,B)) I,
\quad1\le p\le2.
$$
If $f$ is pmi, then
$$
P_f(A,B) \ge I\ \implies \ P_f(A^p,B^p) \ge K(\xi,2p-1)^{-1}\|P_f(A,B)\|_\infty^{1-p} I,
\quad1\le p\le2.
$$

The same statements hold for $P_g$ when $\xi:=\|B\|_\infty/\lambda_{\min}(B)$ and
$g\in\OMD_+^1$ is pmd or pmi.
\end{corollary}

On the other hand, we showed the following result in \cite[p.\ 137, Corollary 5.3.]{FMPS}:
Let $A$ and $B$ be positive invertible operators with $m\leq A, B \leq M$ for some scalars
$0<m<M$, and put $\xi:=M/m$. For any $\alpha >1$ and every $p>1$,
\begin{equation} \label{eq:FMPS}
P_{t^{\alpha}}(A,B) \leq I\ \implies\ P_{t^{\alpha}}(A^p,B^p)
\leq K(\xi^{2p},\alpha) K(\xi,p)^{\alpha}I.
\end{equation}
We remark that in the case of $p=1$, we have
$\lambda_{\min}^{1-p}(P_f(A,B)) K(\xi,2p-1)=1$ in Corollary \ref{C-3.11}, but 
$K(\xi^{2p},\alpha) K(\xi,p)^{\alpha}>1$ in \eqref{eq:FMPS}.\par

\begin{Problem}\label{Q-3.12}\rm
We have shown that the operator perspectives $P_g$ and $P_f$ satisfy the AH type inequality
\eqref{ando-hiai} for all $p\in(0,1]$ when $g\in\OMD_+^1$ and $f\in\OC_+^1$ with $f(0^+)=0$
and $g,f$ are pmi. A natural question is whether the inequality can hold for more general pmi
functions in $\OC_+^1$. A typical example of such pmi functions is $f(t)=wt^2+1-w$ ($0<w<1$).
It seems to us that this $f$ fails to satisfy \eqref{ando-hiai} for $p\in(0,1]$, while we
cannot produce a counter-example.
\end{Problem}

\subsection{Weak log-majorization for matrices}

In this subsection we assume that $\cH$ is \emph{finite-dimensional}, so $B(\cH)$ is
identified with the $n\times n$ matrix algebra with $n=\dim\cH$. Let $A$ and $B$ be $n\times n$
positive semidefinite matrices. Let $\lambda_1(A)\ge\dots\ge\lambda_n(A)$ be the eigenvalues of
$A$ in decreasing order counting multiplicities. The \emph{weak majorization} $A\prec_wB$ says
that $\sum_{i=1}^k\lambda_i(A)\le\sum_{i=1}^k\lambda_i(B)$ for all $k=1,\dots,n$.
The \emph{weak log-majorization} $A\prec_{w\log}B$ means that
\begin{align}\label{F-3.4}
\prod_{i=1}^k\lambda_i(A)\le\prod_{i=1}^k\lambda_i(B),\qquad1\le k\le n,
\end{align}
and the \emph{log-majorization} $A\prec_{\log}B$ means that $A\prec_{w\log B}$ and equality
holds in \eqref{F-3.4} for the last $k=n$, i.e., $\det A=\det B$. Also, the
\emph{log-supermajorization} $A\prec^{w\log}B$ is defined by
$$
\prod_{i=1}^k\lambda_{n+1-k}(A)\ge\prod_{i=1}^k\lambda_{n+1-k}(B),\qquad1\le k\le n.
$$
When $A,B$ are positive definite, $A\prec^{w\log}B$ $\iff$ $A^{-1}\prec_{w\log}B^{-1}$. Note
that $A\prec_{w\log}B$ $\implies$ $A\prec_wB$, and see, e.g., \cite{Bhatia,Hi2} for more about
majorizations for matrices. The notions of (weak) log-majorization and the
log-supermajorization are quite useful to produce matrix norm inequalities for symmetric
(or unitarily invariant) norms (see \cite{Hi2}) and symmetric anti-norms (see \cite{BH}).

For the perspective $P_f$ of a power function $f(t)=t^\alpha$, the standard antisymmetric
tensor power technique (see \cite{Bhatia,AH1}) can be used to obtain log-majorizations from
AH type inequalities, as was done in \cite{AH1} for the weighted matrix geometric means
$A\#_\alpha B$ ($0\le\alpha\le1$). From Corollary \ref{C-3.8} specialized to power functions
with the antisymmetric tensor technique, one can obtain the log-majorization as follows:
For any $\alpha\in[-1,0]\cup[1,2]$,
\begin{align}\label{F-3.5}
P_{t^\alpha}(A^p,B^p)\prec_{\log}P_{t^\alpha}(A,B)^p,\qquad0<p\le1,
\end{align}
or equivalently,
\begin{align}\label{F-3.6}
P_{t^\alpha}(A^q,B^q)^{1/q}\prec_{\log}P_{t^\alpha}(A^p,B^p)^{1/p},\qquad0<q\le p.
\end{align}
In fact, \eqref{F-3.5} and \eqref{F-3.6} for $-1\le\alpha\le0$ have recently been obtained in
\cite{KS}, where the symbol $A\natural_\alpha B$ is used for $P_{t^\alpha}(B,A)$.
Also, \eqref{F-3.6} for $1\le\alpha\le2$ has been given in \cite[(5.2)]{Hi3}.

Even for non-power functions we can obtain the following weak log-majorizations though not
log-majorizations. The weak log-(super)majorizations in \eqref{F-3.21} and
\eqref{F-3.22} are stronger versions of Propositions \ref{P-3.9}, though restricted to
matrices. On the other hand, those in \eqref{F-3.23} and \eqref{F-3.24} are rather
considered as the reverse versions of Proposition \ref{P-3.10} without the generalized
Kantorovich constant. Indeed, \eqref{F-3.23} in particular implies that for every $p\in [1,2]$,
$$
\|P_g(A^p,B^p)\|_\infty\ge\lambda_{\min}\biggl({g(C^p)\over g(C)^p}\biggr)
\|P_g(A,B)\|_\infty^{1-p}\|P_g(A,B)\|_\infty^{2p-1},
$$
while the first inequality for $P_g$ in Proposition \ref{P-3.10} implies that for every
$p\in [1,2]$,
$$
\|P_g(A^p,B^p)\|_\infty\le K(\xi,2p-1)\bigg\|{g(C^p)\over g(C)^p}\bigg\|_\infty
\lambda_{\min}^{1-p}(P_g(A,B))\|P_g(A,B)\|_\infty^{2p-1}.
$$
The above two are in opposite directions. Similarly, \eqref{F-3.24} and the second inequality
in Proposition \ref{P-3.10} give the inequalities for $\lambda_{\min}(A^p,B^p)$ in the
opposite directions.

\begin{proposition}\label{P-3.13}
Let $g\in\OMD_+^1$ and $A,B$ be positive definite matrices. Set $C:=B^{-1/2}AB^{-1/2}$. Then
\begin{align}
&P_g(A^p,B^p)\prec_{w\log}\bigg\|{g(C^p)\over g(C)^p}\bigg\|_\infty
\|P_g(A,B)\|_\infty^{1-p}P_g(A,B)^{2p-1},
\qquad\ \ 1/2\le p\le1, \label{F-3.21}\\
&P_g(A^p,B^p)\prec^{w\log}\lambda_{\min}\biggl({g(C^p)\over g(C)^p}\biggr)
\lambda_{\min}^{1-p}(P_g(A,B))P_g(A,B)^{2p-1},
\quad1/2\le p\le1, \label{F-3.22}\\
&\lambda_{\min}\biggl({g(C^p)\over g(C)^p}\biggr)
\|P_g(A,B)\|_\infty^{1-p}P_g(A,B)^{2p-1}\prec_{w\log}P_g(A^p,B^p),
\quad\ 1\le p\le2, \label{F-3.23}\\
&\bigg\|{g(C^p)\over g(C)^p}\bigg\|_\infty
\lambda_{\min}^{1-p}(P_g(A,B))P_g(A,B)^{2p-1}\prec^{w\log}P_g(A^p,B^p),
\qquad\ 1\le p\le2. \label{F-3.24}
\end{align}

The same statements hold for $P_f$ when $f\in\OC_+^1$ with $f(0^+)=0$.
\end{proposition}

\begin{proof}
Assume that $g\in\OMD_+^1$ and $1/2\le p\le1$. Since $0\le2p-1\le1$, Araki's log-majorization
\cite{Ar} (also \cite{AH1}) implies that
$$
B^{p-{1\over2}}g(C)^{2p-1}B^{p-{1\over2}}\prec_{\log}
(B^{1/2}g(C)B^{1/2})^{2p-1}=P_g(A,B)^{2p-1}.
$$
Combining this with \eqref{F-3.13} shows that
$$
P_g(A,B)\le I\ \implies\ P_g(A^p,B^p)\prec_{\log}
\bigg\|{g(C^p)\over g(C)^p}\bigg\|_\infty P_g(A,B)^{2p-1}.
$$
For any $A,B>0$ apply the above to $\alpha^{-1}A,\alpha^{-1}B$ with
$\alpha:=\|P_g(A,B)\|_\infty$; then \eqref{F-3.21} for $P_g$ follows.
To prove \eqref{F-3.22} for $P_g$, replace $g,A,B$ in \eqref{F-3.21} with $g^*,A^{-1},B^{-1}$;
then we have
$$
P_g(A^p,B^p)^{-1}\prec_{\log}
\bigg\|\biggl({g(C^p)\over g(C)^p}\biggr)^{-1}\bigg\|_\infty
\|P_g(A,B)^{-1}\|_\infty^{1-p}P_g(A,B)^{-(2p-1)},
$$
which is equivalent to \eqref{F-3.22}.

Next, assume that $f\in\OC_+^1$ with $f(0^+)=0$ and $1\le p\le2$. Let
$\widetilde C:=A^{1/2}B^{-1}A^{1/2}$. The inequality in \eqref{F-3.16} yields that
\begin{align}\label{F-3.25}
P_f(A,B)\le I\ \implies\ P_f(A^p,B^p)\le
\bigg\|{f\bigl(\widetilde C^p\bigr)\over f\bigl(\widetilde C\bigr)^p}\bigg\|_\infty
\lambda_{\min}^{1-p}(P_f(A,B))A^{p-{1\over2}}h(\widetilde C)^{2p-1}A^{p-{1\over2}}.
\end{align}
Since $2p-1\ge1$, Araki's log-majorization implies that
\begin{align}\label{F-3.26}
A^{p-{1\over2}}h(\widetilde C)^{2p-1}A^{p-{1\over2}}\succ_{\log}
(A^{1/2}h\bigl(\widetilde C\bigr)A^{1/2}\bigr)^{2p-1}=P_f(A,B)^{2p-1}.
\end{align}
Combining \eqref{F-3.25} and \eqref{F-3.26} gives
$$
P_f(A,B)\le I\ \implies\ \bigg\|{f(C^p)\over f(C)^p}\bigg\|_\infty
\lambda_{\min}^{1-p}(P_f(A,B))P_f(A,B)^{2p-1}
\prec^{w\log}P_f(A^p,B^p),
$$
since $\Big\|{f(\widetilde C^p)\over f(\widetilde C)^p}\Big\|_\infty
=\big\|{f(C^p)\over f(C)^p}\big\|_\infty$. Hence \eqref{F-3.24} for $P_f$ follows by
applying the above to $\alpha^{-1}A,\alpha^{-1}B$ with $\alpha:=\|P_f(A,B)\|_\infty$ (but
the effect of $\alpha$ disappears in this case). Replacing $f,A,B$ in \eqref{F-3.24} for
$P_f$ with $f^*,A^{-1},B^{-1}$, we have \eqref{F-3.23} for $P_f$.

Finally, \eqref{F-3.21} and \eqref{F-3.22} for $P_f$ immediately follow from those for $P_g$,
while \eqref{F-3.23} and \eqref{F-3.24} does from those of $P_g$.
\end{proof}

Proposition \ref{P-3.13} immediately implies the following:

\begin{corollary}\label{C-3.14}
Let $g\in\OMD_+^1$ and $A,B$ be positive definite matrices.
\begin{itemize}
\item[$(${\rm 1}$)$] If $g$ is pmi, then
\begin{align*}
&P_g(A^p,B^p)\prec_{w\log}\|P_g(A,B)\|_\infty^{1-p}P_g(A,B)^{2p-1},
\qquad1/2\le p<1, \\
&\|P_g(A,B)\|_\infty^{1-p}P_g(A,B)^{2p-1}\prec_{w\log}P_g(A^p,B^p),
\qquad1\le p\le2.
\end{align*}
\item[$(${\rm 2}$)$] If $g$ is pmd, then
\begin{align*}
&P_g(A^p,B^p)\prec^{w\log}\lambda_{\min}^{1-p}(P_g(A,B))P_g(A,B)^{2p-1},
\qquad1/2\le p<1, \\
&\lambda_{\min}^{1-p}(P_g(A,B))P_g(A,B)^{2p-1}\prec^{w\log}P_g(A^p,B^p),
\qquad\ 1\le p\le2.
\end{align*}
\end{itemize}

The same statements hold for $P_f$ when $f\in\OC_+^1$ with $f(0^+)=0$ is pmi or pmd.
\end{corollary}

\subsection{Bounds of ${{h(C^p)}\over {h(C)^p}}$}

The bounds $\lambda_{\min}\Bigl({h(C^p)\over h(C)^p}\Bigr)$ 
and $\Big\|{h(C^p)\over h(C)^p}\Big\|_\infty$ repeatedly appear in the inequalities obtained
in Sections 3.1--3.3. Although it might not be easy to compute the values, they can be
estimated for a certain $h$ as follows: 

\begin{proposition}\label{P-3.15}
Assume that $h\in \OM_+^1$ is geometrically convex, i.e., $\log h(e^x)$ is convex on
$(-\infty,\infty)$. Let $C>0$ and set $m:=\lambda_{\min}(C)$ and $M:=\|C\|_\infty$. Then
\begin{align*}
&I\le {{h(C^p)}\over {h(C)^p}} \le \max\left\{
 {{h(m^p)}\over {h(m)^p}},
 {{h(M^p)}\over {h(M)^p}}
\right\}I \quad\mbox{for $p>1$}, \\
&I\ge {{h(C^p)}\over {h(C)^p}} \ge \min\left\{
 {{h(m^p)}\over {h(m)^p}},
 {{h(M^p)}\over {h(M)^p}}
\right\}I \quad\mbox{for $0<p<1$}.
\end{align*}
In particular, if $C\ge I$ (\mbox{resp.}, $C\le I$), then 
$$
I \le  {{h(m^p)}\over {h(m)^p}}I \le {{h(C^p)}\over {h(C)^p}} \le {{h(M^p)}\over {h(M)^p}}I
\quad \left( \mbox{resp.},
I\le {{h(M^p)}\over {h(M)^p}}I \le {{h(C^p)}\over {h(C)^p}} \le {{h(m^p)}\over {h(m)^p}}I
\right)
$$
hold for $p>1$, and all the inequalities above are reversed for $0<p<1$. 
\end{proposition}

This immediately follows from the following:
\begin{lemma}\label{L-3.16}
Let $h\in \OM_+^1$. Then the following conditions are equivalent: 
\begin{itemize}
\item[$(${\rm i}$)$] $t\mapsto {{h(t^p)}\over {h(t)^p}}$ is decreasing on $(0,1)$ 
and is increasing on $(1,\infty)$  for all $p>1$;
\item[$(${\rm ii}$)$] $t\mapsto {{h(t^p)}\over {h(t)^p}}$ is increasing on $(0,1)$ 
and is decreasing on $(1,\infty)$  for all $p\in(0,1)$;
\item[$(${\rm iii}$)$] $h$ is geometrically convex.
\end{itemize}
\end{lemma}

\begin{proof}
Put $f(x):=\log h (e^x)$. Since 
$$
{{\left( {{h(e^{px})}\over {h(e^x)^p}} \right)'} \over {{h(e^{px})}\over {h(e^x)^p}}}
= \left( \log {{h(e^{px})}\over {h(e^x)^p}} \right)'
=p(f'(px)- f'(x)), 
$$
the condition that  $f'$ is increasing is equivalent to each of (i) and (ii). 
\end{proof}

The estimate in Proposition \ref{P-3.15} is applicable to $f\in\OC_+^1$ with $f(0^+)=0$
and $g\in\OMD_+^1$ as well. Indeed, we have $f=th$ and $g=h(t^{-1})=\widetilde f$ for some
$h\in\OM_+^1$ so that, as in \eqref{F-3.14},
$$
{f(C^p)\over f(C)^p}={g(C^{-p})\over g(C^{-1})^p}={h(C^p)\over h(C)^p}.
$$

A study of operator means whose representing functions are geometrically convex
is found in a recent paper \cite{wada3}. An operator mean is called a \emph{geodesic mean}
if it has the representing function $h(t)=\int_0^1t^\alpha\,d\nu(\alpha)$ with a probability
measure $\nu$ on $[0,1]$. As readily verified, such a function $h$ is geometrically convex.
For example, when $h(t)={{t^\alpha+t^{1-\alpha}}\over2}$ with $\alpha\in(0,1)$, note by
Proposition \ref{P-3.15} that
$$
I\le {{h(C^p)}\over {h(C)^p}} \le 
\max\left\{\lim_{t\rightarrow 0^+}{{h(t^p)}\over {h(t)^p}} , 
\lim_{t\rightarrow  \infty}{{h(t^p)}\over {h(t)^p}}\right\}I
=2^{p-1}I
$$
for any $C>0$ and $p>1$.

\subsection{Range of parameter $p$}

We assume that $f$ is a continuous function on $(0,\infty)$ such
that $f>0$ and $f(1)=1$. We denote by $\Lambda (f)$ the set of the parameter $p>0$ for which
$P_f$ satisfies \eqref{ando-hiai}, or equivalently, $P_{f^*}$ satisfies \eqref{ando-hiai2}.
As follows from Theorem \ref{T-3.7}, if $h\in \OM_+^1$ is pmi, then
$\Lambda (h^*)\supseteq [1,\infty)$. Furthermore, when $h\in \OM_+^1\setminus\{1,t\}$ is pmi,
the set $\Lambda (h^*)$ was determined in \cite[Corollary 3.1]{wada2} as follows:
\begin{equation}\label{pmd-ando-hiai}
\Lambda (h^*)=[1,\infty).
\end{equation}

On the other hand, it follows from Theorem \ref{T-3.7} that 
if $h\in\OM_+^1$ is pmi, then $\Lambda (th)\supseteq (0,1]$. 
In this section we shall prove that $\Lambda (th)=(0,1]$ when $h\in\OM_+^1\setminus\{1\}$ is
pmi. 

\begin{proposition}\label{P-3.17}
Assume that $f$ satisfies the following three conditions: 
\begin{itemize}
\item[$(${\rm a}$)$] $\lim_{t\rightarrow 0^+} tf(t) =0$;
\item[$(${\rm b}$)$] $f$ is pmi (resp., pmd);
\item[$(${\rm c}$)$] $f$ is strictly increasing (resp., strictly decreasing).
\end{itemize}
Then $\Lambda(tf)\subseteq(0,1]$ (resp., $\Lambda(tf)\subseteq[1,\infty)$). 
\end{proposition}

The following technical lemma is critical in our proof of this result.

\begin{lemma}\label{L-3.18}
Assume that $f$ satisfies $($a$)$ of Proposition \ref{P-3.17}. If $p\in\Lambda(tf)$, then
$$
f(\lambda a^p +(1-\lambda) b^p ) \le f(\lambda a  +(1-\lambda) b )^p  
$$
holds for all $a,b>0$ and all $\lambda\in [0,1]$. 
\end{lemma}

\begin{proof}
From condition (a) the function $tf$ can extend continuously to $[0,\infty)$ by setting
$(tf)(0):=0$. Assume that $p\in\Lambda(tf)$, i.e., $P_{tf}$ satisfies \eqref{ando-hiai} for
$p$, which is equivalently rewritten as
\begin{align}\label{F-4.1}
\|P_{tf}(A^p,B^p)\|_\infty\le\|P_{tf}(A,B)\|_\infty^p,\qquad A,B>0.
\end{align}
From the definition in \eqref{F-2.4} it is clear that $P_{tf}(A,B)$ is well defined for all
$A\ge0$ and $B>0$. Then the inequality in \eqref{F-4.1} extends to $A\ge0$ and $B>0$, since
$P_{tf}(A+\eps I,B)\to P_{tf}(A,B)$ in the operator norm as $\eps\searrow0$.

Here, for $a,b>0$, we define 
$$ A
=\begin{pmatrix}
\cos^2\theta & \cos\theta\sin\theta \\
\cos\theta\sin\theta & \sin^2\theta
\end{pmatrix}, \qquad 
B=\begin{pmatrix}
a^{-1} & 0 \\ 0 & b^{-1} \end{pmatrix}.
$$
With $c:=\sqrt{a}\cos\theta$ and $d:=\sqrt{b}\sin\theta$, we then compute
$$
B^{-1/2}AB^{-1/2} = \begin{pmatrix}
c^2 \ cd \\ cd \ d^2 \end{pmatrix}
$$
and 
$$
(tf) (B^{-1/2}AB^{-1/2} ) =
{{(tf)(c^2+d^2)}\over {c^2+d^2}}
\begin{pmatrix}
c^2 \ cd \\ cd \ d^2 \end{pmatrix}
=f(c^2+d^2)B^{-1/2}AB^{-1/2}, 
$$
so that
$$P_{tf} (A,B)=f (a \cos^2 \theta +b \sin^2 \theta)
A.$$
In a similar fashion, we have 
$$P_{tf} (A^p,B^p)=P_{tf} (A,B^p)=
f(a^p\cos^2\theta+b^p\sin^2\theta)A.
$$
From \eqref{F-4.1} for $A\ge0$ and $B>0$ it follows that
$$
f(a^p\cos^2\theta+b^p\sin^2\theta)
\le f(a \cos^2 \theta +b \sin^2 \theta )^p.
$$
\end{proof}

\begin{proof}[Proof of Proposition \ref{P-3.17}] 
Suppose that there exists a $p>1$ (resp., $p\in(0,1)$) such that $p\in \Lambda (tf)$. Then 
from the above lemma and the fact that $f$ is pmi (resp., pmd), 
$$
f(\lambda a^p +(1-\lambda) b^p ) \le
f(\lambda a  +(1-\lambda) b )^p \le f((\lambda a  +(1-\lambda) b )^p ) 
$$
holds for all $a,b>0$ and all $\lambda\in [0,1]$.
Since $f$ is strictly increasing (resp., strictly decreasing),  
$$
\lambda a^p +(1-\lambda) b^p \le
(\lambda a  +(1-\lambda) b )^p 
$$
$$
(\mbox{resp.},\ \lambda a^p +(1-\lambda) b^p \ge
(\lambda a  +(1-\lambda) b )^p ) 
$$
holds for all $a,b>0$ and for all $\lambda\in [0,1]$,
contradicting $p>1$ (resp., $p\in(0,1)$). 
\end{proof}

\begin{theorem}\label{T-3.19}
If $h\in OM_+^1\setminus \{1\}$ is pmi, then $\Lambda(th)=\Lambda(h(1/t))=(0,1]$.
\end{theorem}

\begin{proof}
That $\Lambda(th)=(0,1]$ is immediate from Proposition \ref{P-3.17} since
$\Lambda(th)\supseteq(0,1]$ as stated just before the proposition. That
$\Lambda(th)=\Lambda(h(1/t))$ is also immediate from Theorem \ref{T-3.7}.
\end{proof}

\section{Further Ando-Hiai type inequalities}

When $h\in\OM_+^1$ is pmi, Theorem \ref{T-3.7} asserts that $P_{th}$ satisfies
\eqref{ando-hiai} for all $p\in(0,1]$. As noticed in Proposition \ref{P-2.3}, the class
$\{t^nh:h\in\OM_+^1,\,n\in\bN\}$ of positive functions on $(0,\infty)$ is meaningful from
the operator analytical point of view. So the following result is regarded as a natural
continuation of Theorem \ref{T-3.7}.

\begin{theorem}\label{T-4.1}
Let $h\in \OM_+^1$ and $n\in\bN$ with $n\ge2$.
\begin{itemize}
\item[\rm(1)] If $h$ is pmi, then $P_{t^nh}$ satisfies \eqref{ando-hiai} for all
$p\in(0,1/2]$.
\item[\rm(2)] If $h$ is pmd, then $P_{t^nh}$ satisfies \eqref{ando-hiai2} for all
$p\in(0,1/2]$.
\end{itemize}
\end{theorem}

To prove the theorem, we need the following:
 
\begin{lemma}\label{L-4.2}
Let $h\in\OM_+^1$ and let $n$ be a positive integer. Let 
$f=(t^nh)^{[-1]}$ be the inverse function of the function $t\mapsto t^n h(t)$ on $(0,\infty)$. 
Then $f^r$ is in $\OM_+^1$ for any $r\in[0,n]$. 
\end{lemma}

\begin{proof}
First, note that $f=(t^nh)^{[-1]}$ is well defined on $(0,\infty)$. We may prove that $f^n$
is in $\OM_+^1$. When $n=1$, it is known \cite[Lemma 5]{An2} that $(th)^{[-1]}\in\OM_+^1$.
When $n\ge2$, if we put $h_n(t):=h(t^{1/n})$, then $h_n\in\OM_+^1$ and
$$
f^n=\left( (t^nh )^{[-1]}\right)^n
=\left( (th_n\circ t^n)^{[-1]}\right)^n = (th_n)^{[-1]}\in \OM_+^1.
$$
\end{proof}

In the rest of the section we consider a sequence of operator perspectives $(P_n)_{n\ge0}$
defined by 
$$
P_n:=P_{t^n h}.
$$
The following recursive formula of the sequence $P_n$ is easy to verify:
\begin{align*}
P_n(A,B)&=A P_{n-1}(B^{-1},A^{-1}) A = AB^{-1}P_{n-2}(A,B)B^{-1}A,\qquad n\ge2,
\end{align*}
which will be used in the proofs below without reference.

\begin{lemma}\label{L-4.3}
Let $h\in\OM_+^1$ and let $n\in\bN$ with $n\ge2$. If $h$ is pmi, then 
$$
A,B>0,\ \ P_{2n-1}(A,B)\le I\ \implies\ P_{2n-1}(A^p,B^p)\le I
$$
for all $p\in (0,1/2]$.
\end{lemma}

\begin{proof}
The assumption 
$$
I\ge P_{2n-1}(A,B)\ \left(=A P_{2n-2}(B^{-1},A^{-1}) A\right)
$$ 
can be rewritten as 
\begin{align}\label{F-5.1}
(t^{2n-2}h)(A^{1/2} B^{-1}A^{1/2})\le A^{-1}.
\end{align}

We put 
$$
f:=\left( (t^{2n-2}h)^{[-1]}\right)^*,\quad g:=tf\quad\text{and}\quad X:=g^{[-1]}(B).
$$
It follows from Lemma \ref{L-4.2} that $(t^{2n-2}h)^{[-1]}$ and hence $f$ are in $\OM_+^1$.
So, from Lemma \ref{L-4.2} again, $g^{[-1]}$ is also in $\OM_+^1$. Hence, inequality
\eqref{F-5.1} implies that
$$
A^{-1/2} BA^{-1/2}\ge f (A),\quad B\ge g(A) \ \ \text{ and }\ \ X\ge A.
$$ 
 
Here, we shall show  the following inequalities by induction:
\begin{equation}\label{aim1}
P_{2m-1}(A^p,B^p) \le f(A)^{2(n-m)p}
\end{equation}
for $m=1, \ldots ,n$.
When $m=1$, 
\begin{align*}
P_1(A^p,B^p)&=P_{th}(A^p,B^p)=P_{h(1/t)}(B^p,A^p) \\
&\le P_{h(1/t)}(g(A)^p,A^p)=A^p h\left(\left({A\over g(A)}\right)^p\right) \\
&=A^p h\left(\left({1\over {f(A)}}\right)^p\right) \\
&\le \left( A h\left({1\over {f(A)}}\right)\right)^p
=f(A)^{2(n-1)p}.  
\end{align*}
In the above, the latter inequality is derived from the pmi of $h$, and the last equality
follows since
\begin{align}\label{F-5.3}
f^{[-1]}(t)h(t^{-1})=(t^{2n-2}h)^*(t)h(t^{-1})=t^{2n-2}.
\end{align}

If we assume that inequality (\ref{aim1}) holds for $m$ ($< n-1$), then 
\begin{align*}
P_{2m+1}(A^p,B^p)
&=A^pB^{-p} P_{2m-1}(A^p,B^p) B^{-p}A^p \\
&\le A^pB^{-p} f(A)^{2(n-m)p} B^{-p}A^p \\
&\le A^pB^{-p} f(X)^{2(n-m)p} B^{-p}A^p \\
&=A^p \left( {{f(X)^{n-m-1} }\over {X}}\right)^{2p} A^p  \\
&\le A^p \left( {{f(A)^{n-m-1} }\over {A}}\right)^{2p} A^p  =f(A)^{2(n-m-1)p}.
\end{align*}
In the above, the second inequality holds since Lemma \ref{L-4.2} implies that
$$
f^{2(n-m)p}=\left(\left( (t^{2n-2}h)^{[-1]}\right)^{2(n-m)p}\right)^*\in\OM_+^1.
$$
Note that $t/f(t)^{n-m-1}$ is the transpose of $\left((t^{2n-2}h)^{[-1]}\right)^{n-m-1}$
and so $t/f(t)^{n-m-1}\in\OM_+^1$. From this and $2p\le1$ the last inequality in the above
follows. Thus, inequality \eqref{aim1} holds for $m=n$, proving that
$P_{2n-1}(A^p,B^p)\le I$.
\end{proof}

\begin{lemma}\label{L-4.4}
Let $h\in OM_+^1$ and let $n\in\bN$. If $h$ is pmi, then 
$$
A,B>0,\ \ P_{2n}(A,B)\le I\ \implies\ P_{2n}(A^p,B^p)\le I
$$
for all $p\in (0,1/2]$.
\end{lemma}
\begin{proof}
Put
$$
f:=\left( (t^{2n-1}h)^{[-1]}\right)^*,\quad g:=tf\quad \text{and}\quad X:=g^{[-1]}(B).
$$
Then, from Lemma \ref{L-4.2}, $(t^{2n-1}h)^{[-1]}$, $f$ and $g^{[-1]}$ are in $OM_+^1$. 
So the assumption   
$$
I\ge P_{2n}(A,B)\ \left(=A P_{2n-1}(B^{-1},A^{-1}) A\right)
$$ 
implies that
$$A^{-1/2} BA^{-1/2}\ge f (A),\quad B\ge g(A) \ \ \text{ and }\ \ X\ge A .$$ 

Here, we shall show  the following inequalities by induction:
\begin{equation}\label{aim2}
P_{2m}(A^p,B^p) \le f(A)^{2(n-m)p} 
\end{equation}
for $m=1, \ldots ,n$.
When $m=1$, 
\begin{align*}
P_2(A^p,B^p)
&=A^pB^{-p}  (B^p\sigma_h A^p)  B^{-p}A^p  \\
&\le A^pB^{-p} (B^p \sigma_h X^p)  B^{-p}A^p
=A^p \left( {{h \left(\left({X\over B}\right)^p\right) } \over B^p}\right) A^p \\
&\le A^p \left( {{h \left({{X}\over {B}}\right) } \over B}\right)^p A^p
=A^p \left({h\left({1\over f(X)}\right)\over Xf(X)}\right)^p A^p \\
&=A^p\left( {{f(X)^{n-1}}\over X}\right)^{2p}A^p \\
&\le A^p\left( {{f(A)^{n-1}}\over A}\right)^{2p}A^p =f(A)^{2(n-1)p}.
\end{align*}
In the above, the second inequality is due to the pmi of $h$, the fourth equality follows
from $f^{[-1]}(t)h(t^{-1})=t^{2n-1}$ as in \eqref{F-5.3}, and the last inequality follows
since $t/f(t)^{n-1}\in\OM_+^1$ as in the last part of the proof of Lemma \ref{L-4.3}.

If we assume that inequality \eqref{aim2} holds for $m$ ($< n)$, then we can show that
$$
P_{2m+2}(A^p,B^p)\le f(A)^{2(n-m-1)p}
$$
in a similar way to the last paragraph of the proof of Lemma \ref{L-4.3}. Thus, inequality
\eqref{aim2} holds for $m=n$, proving that $P_{2n}(A^p,B^p)\le I$.
\end{proof}

\begin{proof}[Proof of Theorem \ref{T-4.1}]
The first statement (1) is immediate from Lemmas \ref{L-4.3} and \ref{L-4.4}.
Since the adjoint of $t^nh$ is $t^nh^*$, (2) follows as well.
\end{proof}

\begin{corollary}\label{C-4.5}
If $h\in \OM_+^1$ is pmi, then 
$$
(0,1/2]\subseteq \Lambda (t^n h) \subseteq (0,1]
$$
for any integer $n\ge 2$.
\end{corollary}

\begin{proof}
Immediate from Theorem \ref{T-4.1} and Proposition \ref{P-3.17}. 
\end{proof}

Specializing to the power functions $t^\alpha$, the set $\Lambda(t^\alpha)$ of the
parameter $p>0$ for which the AH inequality holds is symmetric at $\alpha=1/2$, since
$\widetilde{t^\alpha}=t^{1-\alpha}$. The $\Lambda(t^\alpha)$ known so far is summarized in
the following:

\begin{proposition}\label{P-4.6}
Let $\alpha \in {\mathbb R}\setminus \{0,1\}$. 
Then $\Lambda(t^\alpha)$ is given as follows:
\begin{itemize}
\item[$(${\rm 1}$)$] $\bigl(0,{\alpha\over2(\alpha-1)}\bigr]\subseteq \Lambda(t^\alpha)
\subseteq (0,1]$
\quad$(\alpha >2)$,
\item[$(${\rm 2}$)$] $\Lambda (t^\alpha) = (0,1]$\quad $(1<\alpha \le 2)$,
\item[$(${\rm 3}$)$] $\Lambda (t^\alpha) = [1,\infty)$ \quad $(0<\alpha <1)$,
\item[$(${\rm 4}$)$] $\Lambda (t^\alpha) = (0,1]$ \quad $(-1\le \alpha <0)$,
\item[$(${\rm 5}$)$] $\bigl(0,{1-\alpha\over-2\alpha}\bigr]\subseteq \Lambda(t^\alpha)
\subseteq (0,1]$ \quad $(\alpha < -1)$.
\end{itemize}
\end{proposition}

\begin{proof}
When $\alpha>2$, Corollary \ref{C-4.5} immediately implies that
$(0,1/2]\subseteq\Lambda(t^\alpha)\subseteq(0,1]$. But a slightly better result that
$\bigl(0,{\alpha\over2(\alpha-1)}\bigr]\subseteq \Lambda(t^\alpha)$ was obtained in
\cite[Corollary 5.2]{Hi3}. Hence we have (1).
Theorem \ref{T-3.19} contains (2) and (4). We have (3) by \cite{AH1} and \cite{wada2}. 
Since $P_{t^\alpha}(A,B)=P_{t^{1-\alpha}}(B,A)$, (5) follows from (1).
\end{proof}

\begin{Remark}\label{R-4.7}\rm
Let $\alpha \in {\mathbb R}\setminus \{0,1\}$. For any $p\in\Lambda(t^\alpha)$ described
in Proposition \ref{P-4.6}, the log-majorization in \eqref{F-3.5} for $P_{t^\alpha}$
is obtained by the standard antisymmetric tensor power technique. Furthermore, the
log-majorization in \eqref{F-3.6} for $P_{t^\alpha}$ holds for any $p,q>0$ with
$q/p\in\Lambda(t^\alpha)$.
\end{Remark}

\begin{Problem}\label{Q-4.8}\rm
An interesting open problem is to determine $\Lambda(t^nh)$ when $n\ge2$ and $h\in\OM_+^1$ is
pmi, in particular, $\Lambda(t^\alpha)$ for $\alpha>2$.
\end{Problem}

The following is a result related to the above problem.

\begin{proposition}\label{P-4.9}
Let $f>0$ be a pmi (resp., pmd) continuous function on $(0,\infty)$. If 
$f$ is not a power function, then $\Lambda (f) \subseteq (0,1]$
(resp., $\Lambda (f) \subseteq [1,\infty)$).
\end{proposition}

\begin{proof}
Since $f$ is pmi (resp. pmd), $f(t^x)\le f(t)^x$ holds for all $t>0$ and
for all $x\in(0,1)$ (resp., $x>1$).  
Assume that there exists a $p>1$ (resp., $p\in(0,1)$) such that 
$p$ is in $\Lambda(f)$. Then from Proposition \ref{P-3.5}, 
$f(t^{p^n x})\le f(t)^{p^n x}$ for all $t>0$ and for all $n\ge 1$ and
$x\in(0,1)$ (resp., $x>1$). This implies that
$$
f(t^x)\le f(t)^x
$$
holds for all $t>0$ and for all $x>0$. So 
$f(t)=f(t^{x\cdot {1\over x}})\le 
f(t^{x})^{1\over x}
\le
f(t)^{x\cdot {1\over x}}=f(t)
$. Thus $f$ must be a power function. This contradicts the assumption. 
\end{proof}

\section{Lie-Trotter formula and norm inequalities}

In this section, applying the Lie-Trotter formula to the AH type inequalities
in Sections 3 and 4, we show operator norm inequalities related to operator means and
operator perspectives. Furthermore, we extend some results in \cite{An1,Ya1} to more general
operator means.

\subsection{Lie-Trotter formula}

In this subsection we present a general Lie-Trotter formula for operator perspectives
associated with positive $C^1$-functions on $(0,\infty)$. Note that most of operator means
and operator perspectives treated in the paper are associated with positive analytic
functions on $(0,\infty)$; so the following Lie-Trotter formula can be applied to them.

\begin{theorem}\label{T-5.1}
Assume that $f$ is a $C^1$ function on $(0,\infty)$ with $f>0$ and $f(1)=1$. Then for every
$A,B>0$,
$$
\lim_{p\to0}P_f(A^p,B^p)^{1/p}=\exp(\alpha\log A+(1-\alpha)\log B)
\quad\mbox{$($in $\|\cdot\|_\infty$ $)$},
$$
where $\alpha:=f'(1)$.
\end{theorem}

The next lemma will be useful to prove the theorem. The lemma seems rather known, but
there seems no suitable reference in the infinite-dimensional setting, so we give a proof for
completeness. We write $B(\cH)^{sa}$ for the set of self-adjoint operators in $B(\cH)$.

\begin{lemma}\label{L-5.2}
Assume that $f$ is a $C^1$ real function on $(0,\infty)$. Let $H\in B(\cH)^{sa}$, and $M(p)$
be a $B(\cH)^{sa}$-valued function on $(-\delta_0,\delta_0)$ for some $\delta_0>0$ such that
$M(0)=0$ and $\|M(p)\|_\infty/|p|\to0$ as $p\to0$. Then there exists a $B(\cH)^{sa}$-valued
function $L(p)$ on $(-\delta,\delta)$ for some $\delta\in(0,\delta_0)$ such that
$$
f(I+pH+M(p))=f(1)I+pf'(1)H+L(p),\qquad p\in(-\delta,\delta),
$$
$$
{\|L(p)\|_\infty\over|p|} \longrightarrow\ 0\quad\mbox{as $p\to0$}.
$$
\end{lemma}

\begin{proof}
Since $\|M(p)\|_\infty/|p|\to0$ as $p\to0$, one can choose an $\alpha>0$ and a
$\delta\in(0,\delta_0)$ such that $\|H+(1/p)M(p)\|_\infty\le\alpha$ for all
$p\in(-\delta,\delta)\setminus\{0\}$ and $\alpha\delta<1$. For each
$p\in(-\delta,\delta)\setminus\{0\}$ let $H+(1/p)M(p)=\int_{-\alpha}^\alpha t\,dE_p(t)$ be
the spectral decomposition of $H+(1/p)M(p)$. Then $f(I+pH+M(p))$ can be given as the spectral
integral as
\begin{align}\label{F-6.2}
f(I+pH+M(p))=\int_{-\alpha}^\alpha f(1+pt)\,dE_p(t).
\end{align}
For any $p$ as above and any $t\in[-\alpha,\alpha]$, by the mean value theorem one has
$$
f(1+pt)=f(1)+ptf'(1+\theta pt)
$$
for some $\theta\in(0,1)$ (depending on $pt$). Set $\phi(p,t):=f'(1+\theta pt)-f'(1)$ for
$p,t$ as above. Then
\begin{align}\label{F-6.3}
f(1+pt)=f(1)+ptf'(1)+pt\phi(p,t),
\end{align}
and from the $C^1$ of $f$ it follows that
\begin{align}\label{F-6.4}
\sup_{|t|\le\alpha}|\phi(p,t)|\ \longrightarrow\ 0\quad\mbox{as $|p|<\delta$, $p\to0$}.
\end{align}
Combining \eqref{F-6.2} and \eqref{F-6.3} gives
$$
f(I+pH+M(p))=f(1)I+pf'(1)\biggl(H+{M(p)\over p}\biggr)
+p\int_{-\alpha}^\alpha t\phi(p,t)\,dE_p(t)
$$
so that
\begin{align*}
&{\|f(I+pH+M(p))-f(1)I-pf'(1)H\|_\infty\over|p|} \\
&\qquad\le|f'(1)|\,{\|M(p)\|_\infty\over|p|}+\sup_{|t|\le\alpha}|t\phi(t,p)|
\ \longrightarrow\ 0\quad\mbox{as $p\to0$}
\end{align*}
due to \eqref{F-6.4}. Hence the result follows by letting
$$
L(p):=f(I+pH+M(p))-f(1)I-pf'(1)H,\qquad p\in(-\delta,\delta).
$$
\end{proof}

\begin{proof}[Proof of Theorem \ref{T-5.1}]
We may prove that
$$
\lim_{p\to0}P_f(e^{pH},e^{pK})^{1/p}=\exp(\alpha H+(1-\alpha)K),
$$
where $H:=\log A$ and $K:=\log B$. From the Taylor expansions of $e^{pH}$ and $e^{pK/2}$ it
is clear that
$$
e^{-pK/2}e^{pH}e^{-pK/2}=I+p(H-K)+M(p)
$$
with $M(p)\in B(\cH)^{sa}$ and $\|M(p)\|_\infty/|p|\to0$ as $p\to0$. Hence by Lemma
\ref{L-5.2} there exists a $B(\cH)^{sa}$-valued function $L(p)$ on $(-\delta,\delta)$ for
some $\delta>0$ such that
$$
f(e^{-pK/2}e^{pH}e^{-pK/2})=I+p\alpha(H-K)+L(p),\qquad p\in(-\delta,\delta),
$$
$$
{\|L(p)\|_\infty\over|p|}\ \longrightarrow\ 0\quad\mbox{as $p\to0$}.
$$
Then we immediately find that
$$
P_f(e^{pH},e^{pK})=I+p(\alpha H+(1-\alpha)K)+\widetilde L(p)
$$
with $\widetilde L(p)\in B(\cH)^{sa}$ for $p\in(-\delta,\delta)$ satisfying
$\|\widetilde L(p)\|_\infty/|p|\to0$ as $p\to0$. By using Lemma \ref{L-5.2} again to the
function $\log$ it follows that there exists a $B(\cH)^{sa}$-valued function $N(p)$ on
$(-\delta',\delta')$ for some $\delta'\in(0,\delta)$ such that
$$
\log P_f(e^{pH},e^{pK})=p(\alpha H+(1-\alpha)K)+N(p),\qquad p\in(-\delta',\delta'),
$$
$$
{\|N(p)\|_\infty\over|p|}\ \longrightarrow\ 0\quad\mbox{as $p\to0$}.
$$
Therefore,
$$
{1\over p}\log P_f(e^{pH},e^{pK})\ \longrightarrow\ \alpha H+(1-\alpha)K
\quad\mbox{(in $\|\cdot\|_\infty$)\quad as $p\to0$},
$$
which yields the required assertion.
\end{proof}

\subsection{Miscellaneous operator norm inequalities}

Assume that $h\in\OM_+^1$ is pmi, and let $n$ be any positive integer. Theorems \ref{T-3.7}
and \ref{T-4.1} say that $P_{t^nh}$ satisfies the AH inequality in \eqref{ando-hiai}
for all $p\in(0,1/2]$. This is equivalently stated as the following operator norm inequality:
For every $A,B>0$,
$$
\|P_{t^nh}(A^p, B^p)\| \leq \|P_{t^nh}(A,B)^p\| \quad \mbox{if $0<p\le1/2$},
$$
which is also equivalently written as
\begin{equation} \label{eq:n3}
\|P_{t^nh}(A^q, B^q)^{1/q}\| \leq \|P_{t^nh}(A^p,B^p)^{1/p}\|
\quad \mbox{if $0<q \leq p/2$}.
\end{equation}
Moreover, Theorem \ref{T-3.7} says also that $P_{h^*}$ satisfies \eqref{ando-hiai} for all
$p\in[1,\infty)$, which is equivalently stated as
\begin{align}\label{F-6.5}
\|P_{h^*}(A^p,B^p)^{1/p}\|_\infty \le \|P_{h^*}(A^q,B^q)^{1/q}\|_\infty
\quad \mbox{if $0<q\le p$}.
\end{align}
Since $(t^nh)'(1)=n+h'(1)$ for any $n\in\bN$, the next corollary immediately follows by
letting $q\searrow0$ in \eqref{eq:n3} and \eqref{F-6.5} due to Theorem \ref{T-5.1}.

\begin{corollary}\label{C-5.3}
Assume that $h\in\OM_+^1$ is pmi, and let $\alpha:=h'(1)$. Then for every $A,B>0$ and all
$p>0$,
\begin{align}
\|P_{h^*}(A^p,B^p)^{1/p}\|_\infty\le\|\exp(\alpha\log A+(1-\alpha)\log B)\|_\infty,&
\label{Log-Euclid1}\\
\|\exp((n+\alpha)\log A+(1-n-\alpha)\log B)\|_\infty\le\|P_{t^nh}(A^p,B^p)^{1/p}\|_\infty,&
\quad n\in\bN. \label{Log-Euclid2}
\end{align}
\end{corollary}

For $\alpha\in[0,1]$ the operator $\exp(\alpha\log A+(1-\alpha)\log B)$ inside the right-hand
side of \eqref{Log-Euclid1} is called the ($\alpha$-weighted) \emph{Log-Euclidean mean} of
$A,B>0$. Since $\|e^X\|\le1$ is equivalent to $X\le0$ for $X\in B(\cH)^{sa}$,
Corollary \ref{C-5.3} also implies the following:

\begin{corollary}\label{C-5.4}
Let $h$ and $\alpha$ be as in Corollary \ref{C-5.3}. Then for any $A,B>0$ and any $n\in\bN$,
\begin{align*}
\alpha \log A+(1-\alpha)\log B\le0\ &\implies\ P_{h^*}(A,B)\le I,
\ \ \mbox{i.e.},\ \ B\sigma_{h^*}A\le I, \\
P_{t^nh}(A,B)\le I\ &\implies\ \log A\le{n+\alpha-1\over n+\alpha}\log B.
\end{align*}
\end{corollary}

Specializing to the power functions $t^\alpha$ we state the following:

\begin{corollary}\label{C-5.5}
\begin{itemize}
\item[\rm(1)] For every $\alpha>1$ and positive invertible operators $A,B$,
\begin{align*}
\big\|(B^{-q}\#_{\alpha\over2\alpha-1}A^q)^{2\alpha-1\over q}\big\|_\infty
&\le\|\exp(\alpha\log A+(1-\alpha)\log B)\|_\infty \\
&\le\|P_{t^\alpha}(A^p,B^p)^{1/p}\|_\infty,\qquad p,q>0.
\end{align*}
\item[\rm(2)] For every $\alpha>1$ and positive definite matrices $A,B$,
\begin{align}
(B^{-q}\#_{\alpha\over2\alpha-1}A^q)^{2\alpha-1\over q}
&\prec_{\log}\exp(\alpha\log A+(1-\alpha)\log B) \nonumber\\
&\prec_{\log}P_{t^\alpha}(A^p,B^p)^{1/p},\qquad p,q>0. \label{Log-Euclid3}
\end{align}
\end{itemize}
\end{corollary}

\begin{proof}
(1)\enspace
Let $\alpha>1$. Since
$$
\|\exp(\alpha\log A+(1-\alpha)\log B\|_\infty
=\bigg\|\exp\biggl({\alpha\over2\alpha-1}\log A
+{\alpha-1\over2\alpha-1}\log B^{-1}\biggr)\bigg\|_\infty^{2\alpha-1},
$$
the first inequality is a rewriting of \eqref{Log-Euclid1} for
$h(t)=t^{\alpha\over2\alpha-1}=h^*(t)$. The second is obvious from \eqref{Log-Euclid2} by
putting $h(t)=t^{\alpha-n}$ where $n\le\alpha<n+1$.

(2) is an immediate consequence of (1) by the antisymmetric tensor power technique as
mentioned in Section 3.3. (In fact, the first log-majorization in \eqref{Log-Euclid3} is
essentially in \cite[Corollary 2.3]{AH1}.)
\end{proof}

The second log-majorization in \eqref{Log-Euclid3} for $1<\alpha\le2$ was recently shown in
\cite[Theorem 4.4]{KS} and that for $\alpha\ge2$ follows from \cite[Corollary 5.2]{Hi3}.

We have the following simple characterization for operator perspectives to satisfy the
operator norm inequality such as \eqref{Log-Euclid1} or \eqref{Log-Euclid2}. (A related
result in a more general setting when $f\in\OM_+^1$ is found in \cite[Corollary 4.18]{Hi1}.)

\begin{proposition}\label{P-5.6}
Let $f>0$ be a continuous function on $(0,\infty)$.
\begin{itemize}
\item[\rm(1)] For each $\alpha\in[0,1]$ the following conditions are equivalent:
\begin{itemize}
\item[\rm(i)] $f(t)\le t^\alpha$ for all $t>0$;
\item[\rm(ii)] $\|P_f(A,B)\|_\infty\le\|\exp(\alpha\log A+(1-\alpha)\log B)\|_\infty$
for all $A,B>0$;
\item[\rm(iii)] $\|P_f(A^p,B^p)^{1/p}\|_\infty\le\|\exp(\alpha\log A+(1-\alpha)\log B)\|_\infty$
for all $A,B>0$ and all $p>0$;
\item[\rm(iv)] $P_f(A^p,B^p)^{1/p}\prec_{w\log}
\exp(\alpha\log A+(1-\alpha)\log B)$ for all positive definite matrices $A,B$ and all $p>0$.
\end{itemize}
\item[\rm(2)] For each $\alpha\in(-\infty,0]\cup[1,\infty)$ the following conditions are
equivalent:
\begin{itemize}
\item[\rm(i)$'$] $f(t)\ge t^\alpha$ for all $t>0$;
\item[\rm(ii)$'$] $\|\exp(\alpha\log A+(1-\alpha)\log B)\|_\infty\le
\|P_f(A,B)\|_\infty$ for all $A,B>0$;
\item[\rm(iii)$'$] $\|\exp(\alpha\log A+(1-\alpha)\log B)\|_\infty\le
\|P_f(A^p,B^p)^{1/p}\|_\infty$ for all $A,B>0$ and all $p>0$;
\item[\rm(iv)$'$] $\exp(\alpha\log A+(1-\alpha)\log B)\prec_{w\log}
P_f(A^p,B^p)^{1/p}$ for all positive definite matrices $A,B$ and all $p>0$.
\end{itemize}
\end{itemize}
\end{proposition}

\begin{proof}
Since the proofs of (1) and (2) are similar, we give only the proof of (2).
Moreover, we may assume that $\alpha\ge1$, since the case $\alpha\le-1$ follows
from the case $\alpha\ge1$ by replacing $f$, $\alpha$ with $\widetilde f$, $1-\alpha$.

(iii)$'$\,$\implies$\,(ii)$'$ is obvious and (ii)$'$ $\implies$ (i)$'$ is easy by taking
$A=tI$ and $B=I$.

(i)$'$\,$\implies$\,(iii)$'$.\enspace
By (i)$'$ and \eqref{eq:n3} for $h(t)=t^{\alpha-n}$ where $n\le\alpha<n+1$, one has
$$
\|P_f(A^p,B^p)\|_\infty^{1/p}\ge\|P_{t^\alpha}(A^p,B^p)\|_\infty^{1/p}
\ge\|P_{t^\alpha}(A^q,B^q)\|_\infty^{1/q},\quad0<q\le p/2.
$$
By the Lie-Trotter formula as $q\searrow0$, (iii)$'$ follows.

(i)$'$\,$\implies$\,(iv)$'$.\enspace
Let $A,B$ be $n\times n$ positive definite matrices. By the antisymmetric tensor power
technique again, from (i)$'$ and \eqref{eq:n3} one has for any $k=1,\dots,n$,
$$
\prod_{i=1}^k\lambda_i^{1/p}(P_f(A^p,B^p))\ge
\prod_{i=1}^k\lambda_i^{1/p}(P_{t^\alpha}(A^p,B^p))
\ge\prod_{i=1}^k\lambda_i^{1/q}(P_{t^\alpha}(A^q,B^q)),\quad0<q\le p/2.
$$
Letting $q\searrow0$ gives (iv)$'$.
\end{proof}

\begin{Remark}\label{R-5.7}\rm
From Corollary \ref{C-5.3} and Proposition \ref{P-5.6} we notice that if $h\in\OM_+^1$ is pmd,
then $h(t)\le t^\alpha$ where $\alpha=h'(1)$ ($\in[0,1]$), which was recently pointed out in
\cite[Section 5]{wada3}. Moreover it was shown in \cite{wada3} that there is an $h\in\OM_+^1$
such that $h(t)\le t^\alpha$ for some $\alpha\in[0,1]$ but $h(t^p)\not\le h(t)^p$ for any
$p>1$ (hence $h$ is not pmd). We thus see that for $h\in\OM_+^1$, the AH inequality
$$
\|(A^p\sigma_hB^p)^{1/p}\|_\infty\le\|A\sigma_hB\|_\infty,\qquad p\ge1,
$$
is equivalent to the pmd of $h$, while the weaker inequality
$$
\|(A^p\sigma_hB^p)^{1/p}\|_\infty\le\|\exp((1-\alpha)\log A+\alpha\log B)\|_\infty,
\qquad p>0,
$$
is equivalent to $h(t)\le t^\alpha$, where $\alpha=h'(1)$.
\end{Remark}

The next corollary may be considered as the operator perspective version of 
\cite[Theorem 1]{An1} (also \cite[Theorem 1]{Ya1}).

\begin{corollary}\label{C-5.8}
Let $n\in {\mathbb N}$ and $h \in \OM_+^1$ be pmi. Set 
$\alpha:=n+h'(1)$. Then  for any $A,B>0$, the following conditions are equivalent:
\begin{itemize}
\item[$(${\rm i}$)$] $\alpha \log A + (1-\alpha) \log B < 0$; 
\item[$(${\rm ii}$)$] $\|P_{t^n h}(A^p,B^p) \|_\infty<1$ for some $p>0$; 
\item[$(${\rm iii}$)$] $\|P_{t^\alpha }(A^p,B^p) \|_\infty<1$ for some $p>0$; 
\item[$(${\rm iv}$)$]  there exists an $r\in (0,1)$ such that 
$B^{-p}\#_{\alpha\over2\alpha-1}A^p \le r^p I$
holds for all $p>0$.
\end{itemize}
\end{corollary}

\begin{proof}
(i)\,$\implies$\,(ii) is immediate from  Theorem \ref{T-5.1}.
From Theorem \ref{T-5.1} and \eqref{eq:n3}, (ii) implies that
\begin{align*}
\|\exp(\alpha \log A + (1-\alpha) \log B )\|_\infty 
&=\lim_{p/2\ge q\searrow0} \|(P_{t^n h}(A^q,B^q))\|_\infty^{1/q} \\
&\le \|P_{t^n h}(A^p,B^p) \|_\infty^{1/p}<1.
\end{align*}
Hence (i)\,$\iff$\,(ii), and (i)\,$\iff$\,(iii) is seen in a similar way.
(i)\,$\implies$\,(iv) is immediate from Corollary \ref{C-5.5}\,(1). Finally,
(iv)\,$\implies$\,(i) follows from Theorem \ref{T-5.1} as
$$
\bigg\|\exp\left( {{\alpha}\over {2\alpha-1}} \log A + {{\alpha-1}\over {2\alpha-1}}\log B^{-1}    
\right)
\bigg\|_\infty \le 
\sup_{p>0}\|(B^{-p}\#_{\alpha\over2\alpha-1}A^p)^{1/p}\|_\infty\le r<1.
$$ 
\end{proof}

In the rest of the subsection, we extend \cite[Theorem 1]{An1} and \cite[Theorem 1]{Ya1} for
the (weighted) operator geometric means to general operator means having the pmd (or pmi)
representing function.

\begin{proposition}\label{P-5.9}
Let $\alpha\in (0,1)$ and $\PMD_\alpha^1$ be the set of all $h\in \OM_+^1$ such that $h$ is
pmd and $h'(1)=\alpha$. Then for any $A,B>0$ the following conditions are equivalent:
\begin{itemize}
\item[$(${\rm i}$)$] $(1-\alpha)\log A+\alpha \log B \le 0$; 
\item[$(${\rm ii}$)$] $x\mapsto A^x\sigma_hB^x$ is a decreasing map from $[0,\infty)$ into
$B(\cH)^{++}$  
for all $h\in \PMD_\alpha^1$;
\item[$(${\rm iii}$)$] $x\mapsto A^x\sigma_hB^x$ is a decreasing map from $[0,\infty)$ into
$B(\cH)^{++}$ for some $h\in \PMD_\alpha^1$;
\item[$(${\rm iv}$)$]  $x\mapsto A^x\#_\alpha B^x$ is a decreasing map from $[0,\infty)$ into
$B(\cH)^{++}$. 
\end{itemize}
\end{proposition}

\begin{proof}
(i)\,$\implies$\,(ii).\enspace
From Corollary \ref{C-5.3}, 
$$
A^x\sigma_h B^x\le 
\|P_{h}(B^x,A^x)\|_\infty\le\|\exp(\alpha\log B+(1-\alpha)\log A) \|_\infty^x \le 1,
\qquad x>0.
$$
So, if $0< p<q$, then
it follows from \eqref{F-7.7} that
$$
A^q\sigma_hB^q
\le \|A^{p}\sigma_hB^{p}\|_\infty^{{q\over p}-1}(A^{p}\sigma_hB^{p}) \\
\le A^{p}\sigma_hB^{p}\le I=A^0 \sigma_h B^0. 
$$

(ii)\,$\implies$\,(iii) is obvious. 

(iii)\,$\implies$\,(iv).\enspace 
Since $A^{x/r}\sigma_hB^{x/r}\le I$ for any $x\ge0$ and $r>1$, it follows from
\cite[Proposition 6.2]{HSW} that
\begin{align}\label{F-6.7}
A^x\sigma_{h_{[1/r]}}B^x\le I,
\end{align}
where $h_{[1/r]}(t):=h(t^{1/r})^{r}$. Here, as a special case of Theorem \ref{T-5.1},
note that
$$
\lim_{r\to\infty}h(X^{1/r})^r=\lim_{r\to\infty}(I^{1/r}\sigma_hX^{1/r})^r=X^\alpha,
\qquad X\in B(\cH)^{++}.
$$
Therefore, taking the limit of \eqref{F-6.7} as $r\rightarrow \infty$ gives 
$A^x\#_\alpha B^x \le I$ for all $x\ge0$. By a similar argument to the proof of
(i)\,$\implies$\,(ii), (iv) follows. 

(iv)\,$\implies$\,(i).\enspace
From Theorem \ref{T-5.1}, 
$$\|\exp((1-\alpha)\log A+\alpha \log B)\|_\infty
=
\lim_{x\rightarrow 0+}
\|A^x\#_\alpha B^x\|_\infty^{1/x}\le I. 
$$
\end{proof}

Since $(A^x\sigma_hB^x)^{-1}=A^{-x}\sigma_{h^*}B^{-x}$, 
Proposition \ref{P-5.9} is rephrased as follows:

\begin{corollary}\label{C-5.10}
Let $\alpha\in (0,1)$ and $\PMI_\alpha^1$ be the set of all $h\in \OM_+^1$ such that $h$
is pmi and $h'(1)=\alpha$. 
Then for any $A,B>0$ the following conditions are equivalent:
\begin{itemize}
\item[$(${\rm i}$)$] $(1-\alpha) \log A + \alpha \log B \ge0$; 
\item[$(${\rm ii}$)$] $x\mapsto A^x\sigma_hB^x$ is an increasing map from $[0,\infty)$ into
$B(\cH)^{++}$ for all $h\in \PMI_\alpha^1$;
\item[$(${\rm iii}$)$] $x\mapsto A^x\sigma_hB^x$ is an increasing map from $[0,\infty)$ into
$B(\cH)^{++}$ for some $h\in \PMI_\alpha^1$;
\item[$(${\rm iv}$)$]  $x\mapsto A^x\#_\alpha B^x$ is an increasing map from $[0,\infty)$ into
$B(\cH)^{++}$. 
\end{itemize}
\end{corollary}

\section{Extension of operator perspectives to non-invertible operators}

Our main concern in this section is the extension of operator perspectives $P_f$ on
$B(\cH)^{++}\times B(\cH)^{++}$ to $B(\cH)^+\times B(\cH)^+$, thus extending
some inequalities in Section 3 to non-invertible operators. A natural way to extend $P_f$
to $B(\cH)^+\times B(\cH)^+$ is to consider the limit
\begin{align}\label{F-8.1}
\lim_{\eps\searrow0}P_f(A+\eps I,B+\eps I)\qquad\mbox{(SOT)}
\end{align}
for $A,B\ge0$ as long as the limit exists in SOT (the strong operator topology). The extension
problem like this for operator perspectives has not been discussed so far except those in
\cite{HM} in the finite-dimensional case.

We shall restrict our consideration to the case where $f$ is operator convex on $(0,\infty)$
but $f$ is not assumed to be positive. The next proposition characterizes when
the limit in \eqref{F-8.1} exists unconditionally.

\begin{proposition}\label{P-6.1}
Let $f$ be an operator convex function on $(0,\infty)$. Then the following conditions are
equivalent:
\begin{itemize}
\item[$(${\rm i}$)$] the limit in \eqref{F-8.1} exists in $B(\cH)$ for all $A,B\in B(\cH)^+$;
\item[$(${\rm ii}$)$] $f(0^+)<\infty$ and $f'(\infty)<\infty$;
\item[$(${\rm iii}$)$] there exist $\alpha,\beta\in\bR$ and $h\in\OM_+\cup\{0\}$ such that
$f(t)=\alpha+\beta t-h(t)$ for all $t>0$.
\end{itemize}
\end{proposition}

\begin{proof}
(i)\,$\implies$\,(ii).\enspace
For $A=aI$ and $B=bI$ with scalars $a,b\ge0$, we have
$$
P_f(A+\eps I,B+\eps I)=(b+\eps)f\biggl({a+\eps\over b+\eps}\biggr)I.
$$
When $a=0$ and $b=1$, $(1+\eps)f\bigl({\eps\over1+\eps}\bigr)\to f(0^+)$ as $\eps\searrow0$.
When $a=1$ and $b=0$, $\eps f\bigl({1+\eps\over\eps}\bigr)
=(1+\eps){\eps\over1+\eps}f\bigl({1+\eps\over\eps}\bigr)\to f'(\infty)$ as $\eps\searrow0$.
Hence (i) implies (ii).

(ii)\,$\implies$\,(iii) was shown in \cite[Theorem 8.4]{HMPB}.

(iii)\,$\implies$\,(i).\enspace
Assume (iii). For every $A,B>0$ one has
\begin{align}\label{F-8.2}
P_f(A,B)=\alpha B+\beta A-B\sigma_hA,
\end{align}
where $\sigma_h$ is the operator connection corresponding to $h$ (in Kubo-Ando's sense).
Hence (i) follows from the downward continuity of the operator connection \cite{KA}.
\end{proof}

When the equivalent conditions of Proposition \ref{P-6.1} are satisfied, one can write the
extension of $P_f$ to $B(\cH)^+\times B(\cH)^+$ as \eqref{F-8.2} for $A,B\ge0$, which is
indeed the extension of $P_f$ for $A,B>0$. Thus, the extended operator perspective $P_f$ in
this case is essentially the minus of the operator connection $\sigma_h$. Moreover, if
$A_n\searrow A$ and $B_n\searrow B$ in $B(\cH)^+$, then $P_f(A,B)=\lim_nP_f(A_n,B_n)$ in SOT. 

Here we recall the well-known fact that if $A,B\ge0$ and $A\le cB$ for some $c>0$, then there
is a unique positive operator $W$ ($\le cI$) such that $W(I-s(B))=0$ and $A=B^{1/2}WB^{1/2}$,
where $s(B)$ is the support projection of $B$ (i.e., the orthogonal projection onto the the
closure of the range of $B$). We denote this $W$ by $D(A/B)$ to specify its
dependence on $A,B$. Clearly, we have $D(A/B)=B^{-1/2}AB^{-1/2}$ whenever $B>0$.

The next two theorems are our main results of the section on extension of operator
perspectives $P_f$.

\begin{theorem}\label{T-6.2}
Let $f$ be an operator convex function on $(0,\infty)$. Then the following conditions are
equivalent:
\begin{itemize}
\item[$(${\rm i}$)$] the limit in \eqref{F-8.1} exists for every $A,B\in B(\cH)^+$ such that
$A\le cB$ for some $c>0$;
\item[$(${\rm ii}$)$] $f(0^+)<\infty$.
\end{itemize}

In this case, for every $A,B$ as in $(${\rm i}$)$,
\begin{align}\label{F-8.3}
\lim_{\eps\searrow0}P_f(A+\eps I,B+\eps I)=B^{1/2}f(D(A/B))B^{1/2}\quad\mbox{$(${\rm SOT}$)$},
\end{align}
where $f$ extends to $[0,\infty)$ by $f(0)=f(0^+)$.
\end{theorem}

\begin{proof}
(i)\,$\implies$\,(ii).\enspace
Take $A=0$ and $B=I$; then $f(0^+)<\infty$ follows as in the proof of (i)\,$\implies$\,(ii)
of Proposition \ref{P-6.1}.

(ii)\,$\implies$\,(i).\enspace
Assume that $f(0^+)<\infty$. Then it is known \cite[Theorem 8.1]{HMPB} that $f$ has the
integral expression
$$
f(t)=\alpha+\beta t+\gamma t^2
+\int_{(0,\infty)}\biggl({t\over1+s}-{t\over t+s}\biggr)\,d\mu(s),
\qquad t\in (0,\infty),
$$
where $\alpha,\beta\in\bR$ (note that $\alpha=f(0^+)$), $\gamma\ge0$ and $\mu$ is a positive
measure on $(0,\infty)$ satisfying $\int_{(0,\infty)}(1+s)^{-2}\,d\mu(s)<\infty$. Set
$$
\phi_s(t):={t\over1+s}-{t\over t+s},\qquad t\in(0,\infty).
$$
We can write for $\eps>0$
\begin{align}
P_f(A+\eps I,B+\eps I)&=\alpha(B+\eps I)+\beta(A+\eps I)
+\gamma P_{t^2}(A+\eps I,B+\eps I) \nonumber\\
&\qquad+\int_{(0,\infty)}P_{\phi_s}(A+\eps I,B+\eps I)\,d\mu(s). \label{F-8.4}
\end{align}
Let $A,B\ge0$ with $A\le cB$ for some $c>0$. We may assume that $c\ge1$. For any
$\eps>0$, since $(ct+\eps)/(t+\eps)\le c$ for all $t\ge0$, one has
\begin{align*}
(B+\eps I)^{-1/2}(A+\eps I)(B+\eps I)^{-1/2}
&\le(B+\eps I)^{-1/2}(cB+\eps I)(B+\eps I)^{-1/2} \\
&=(cB+\eps I)(B+\eps I)^{-1}\le cI,
\end{align*}
so that the spectrum of $(B+\eps I)^{-1/2}(A+\eps I)(B+\eps I)^{-1/2}$ is in $[0,c]$.
Note that
$$
\phi_s'(t)={1\over1+s}-{s\over(t+s)^2}={t^2+2st-s\over(1+s)(t+s)^2}
$$
and the solution of $\phi_s'(t)=0$ for $t>0$ is $t=\sqrt{s+s^2}-s<1$, from which one has
$$
\phi_s(\sqrt{s+s^2}-s)\le\phi_s(t)\le\phi_s(c),\qquad t\in[0,c].
$$
A direct computation gives
$$
-(1+s)^2\phi_s(\sqrt{s+s^2}-s)=(1+s)\bigl(\sqrt{1+s}-\sqrt s\bigr)^2
={1+s\over\bigl(\sqrt{1+s}+\sqrt s\bigr)^2}\le1,
$$
and hence $\phi_s(\sqrt{s+s^2}-s)\ge-1/(1+s)^2$. Therefore,
$$
-{1\over(1+s)^2}\,I\le\phi_s((B+\eps I)^{-1/2}(A+\eps I)(B+\eps I)^{-1/2})
\le\phi_s(c) I,
$$
so that for any $\eps\in(0,1)$ one has
$$
-{1\over(1+s)^2}(\|B\|_{\infty}+1)\le P_{\phi_s}(A+\eps I,B+\eps I)
\le\phi_s(c)(\|B\|_{\infty}+1).
$$
Now, suppose that the following limits exist:
\begin{align}
\overline P_{t^2}(A,B)&:=\lim_{\eps\searrow0}P_{t^2}(A+\eps I,B+\eps I)\quad(SOT),
\label{F-8.5}\\
\overline P_{\phi_s}(A,B)&:=\lim_{\eps\searrow0}P_{\phi_s}(A+\eps I,B+\eps I)\quad(SOT),
\qquad s\in(0,\infty). \label{F-8.6}
\end{align}
Then, since $\int_{(0,\infty)}(1+s)^{-2}\,d\mu(s)<\infty$ and
$\int_{(0,\infty)}\phi_s(c)\,d\mu(s)<\infty$, it follows from the Lebesgue convergence
theorem that
\begin{align}\label{F-8.7}
\lim_{\eps\searrow0}\int_{(0,\infty)}P_{\phi_s}(A+\eps I,B+\eps I)\,d\mu(s)
=\int_{(0,\infty)}\overline P_{\phi_s}(A,B)\,d\mu(s).
\end{align}
From \eqref{F-8.4}, \eqref{F-8.5} and \eqref{F-8.7} we obtain
\begin{align}
&\lim_{\eps\searrow0}P_f(A+\eps I,B+\eps I) \nonumber\\
&\qquad=\alpha B+\beta A+\gamma\overline P_{t^2}(A,B)
+\int_{(0,\infty)}\overline P_{\phi_s}(A,B)\,d\mu(s)\quad(SOT), \label{F-8.8}
\end{align}
and the limit in \eqref{F-8.1} exists.

Thus, it remains to prove the existence of the limits in \eqref{F-8.5} and \eqref{F-8.6}.
Since $A\le cB$, we have a bounded operator $V$ with $\|V\|\le c^{1/2}$ such that
$V(I-s(B))=0$ and $A^{1/2}=VB^{1/2}=B^{1/2}V^*$, so $W:=V^*V=D(A/B)$. We write
\begin{align*}
P_{t^2}(A+\eps I,B+\eps I)
&=(A+\eps I)(B+\eps I)^{-1}(A+\eps I) \\
&=A(B+\eps I)^{-1}A+\eps A(B+\eps I)^{-1}+\eps(B+\eps I)^{-1}A \\
&\qquad+\eps^2(B+\eps I)^{-1}I.
\end{align*}
Let $B=\int_0^{\|B\|_{\infty}}\lambda\,dE_\lambda$ is the spectral decomposition. For any
$\xi\in\cH$ note that
\begin{align*}
\|\eps A(B+\eps I)^{-1}\xi\|^2&=\|\eps B^{1/2}WB^{1/2}(B+\eps I)^{-1}\xi\|^2 \\
&\le c^2\|B\|_{\infty}\|\eps B^{1/2}(B+\eps I)^{-1}\eps\|^2 \\
&=c^2\|B\|_{\infty}\int_0^{\|B\|_{\infty}}{\eps^2\lambda\over(\lambda+\eps)^2}
\,d\|E_\lambda\xi\|^2.
\end{align*}
Since $\eps^2\lambda/(\lambda+\eps)^2\le1$ for all $\lambda\ge0$, $\eps\in(0,1)$, and
$\eps^2\lambda/(\lambda+\eps)^2\to0$ for any $\lambda\ge0$ as $\eps\searrow0$, it follows
from the bounded convergence theorem that $\|\eps A(B+\eps I)^{-1}\xi\|\to0$ as
$\eps\searrow0$, so $\eps A(B+\eps I)^{-1}\to0$ in SOT as $\eps\searrow0$. Similarly,
$\eps(B+\eps I)^{-1}A\to0$ in SOT, and $\eps^2(B+\eps I)^{-1}\to0$ is immediate.
Moreover, we write
$$
A(B+\eps I)^{-1}A=A^{1/2}VB^{1/2}(B+\eps I)^{-1}B^{1/2}V^*A^{1/2}
=A^{1/2}VB(B+\eps I)^{-1}V^*A^{1/2}.
$$
Since $B(B+\eps I)^{-1}\to s(B)$ in SOT as $\eps\searrow0$, it follows that
$A(B+\eps I)^{-1}A$ converges in SOT to
$$
A^{1/2}Vs(B)V^*A^{1/2}=B^{1/2}V^*VV^*VB^{1/2}=B^{1/2}W^2B^{1/2}.
$$
Hence \eqref{F-8.5} holds as
\begin{align}\label{F-8.9}
\lim_{\eps\searrow0}P_{t^2}(A+\eps I,B+\eps I)=B^{1/2}W^2B^{1/2}\quad\mbox{(SOT)}.
\end{align}
To prove \eqref{F-8.6}, set $h_s(t):=t/(t+s)$ for $t\in(0,\infty)$. Since $h_s\in\OM_+$,
we write
$$
P_{\phi_s}(A+\eps I,B+\eps I)
={1\over1+s}(A+\eps I)-(B+\eps I)\sigma_{h_s}(A+\eps I),
$$
where $\sigma_{h_s}$ is the operator connection corresponding to $h_s$. Hence \eqref{F-8.6}
holds as
\begin{align}\label{F-8.10}
\lim_{\eps\searrow0}P_{\phi_s}(A+\eps I,B+\eps I)={1\over1+s}\,A-B\sigma_{h_s}A\quad
\mbox{(SOT)}.
\end{align}
Thus, (i) has been shown, and from \eqref{F-8.8}--\eqref{F-8.10} the limit in \eqref{F-8.1}
is equal to
\begin{align}\label{F-8.11}
\alpha B+\beta A+\gamma B^{1/2}W^2B^{1/2}
+\int_{(0,\infty)}\biggl({1\over1+s}\,A-B\sigma_{h_s}A\biggr)\,d\mu(s).
\end{align}

Next, to show the latter assertion of the theorem, we see that for any $h\in\OM_+$,
\begin{align}\label{F-8.12}
B\sigma_hA=B^{1/2}h(W)B^{1/2}.
\end{align}
Indeed, we have
\begin{align*}
B\sigma_hA&=\lim_{\eps\searrow0}(B+\eps I)\sigma_hA \\
&=\lim_{\eps\searrow0}(B+\eps I)^{1/2}
h((B+\eps I)^{-1/2}A(B+\eps I)^{1/2})(B+\eps I)^{1/2} \\
&=\lim_{\eps\searrow0}(B+\eps I)^{1/2}
h\bigl((B+\eps I)^{-1/2}B^{1/2}WB^{1/2}(B+\eps I)^{1/2}\bigr)(B+\eps I)^{1/2}.
\end{align*}
Since $B^{1/2}(B+\eps I)^{-1/2}\to s(B)$ in SOT as $\eps\searrow0$,
$$
(B+\eps I)^{-1/2}B^{1/2}WB^{1/2}(B+\eps I)^{1/2}\ \longrightarrow\ s(B)Ws(B)=W\quad
\mbox{(SOT)}.
$$
From the SOT continuity of the functional calculus $X\in B(\cH)^+\mapsto h(X)$, it follows
that
$$
h\bigl((B+\eps I)^{-1/2}B^{1/2}WB^{1/2}(B+\eps I)^{1/2}\bigr)\ \longrightarrow\ h(W)
\quad\mbox{(SOT)}.
$$
Moreover, since $(B+\eps I)^{1/2}\to B^{1/2}$ in $\|\cdot\|_\infty$, \eqref{F-8.12} follows.
Thus, \eqref{F-8.11} is equal to
\begin{align*}
&\alpha B+\beta B^{1/2}WB^{1/2}+B^{1/2}W^2B^{1/2} \\
&\qquad\quad+\int_{(0,\infty)}
\biggl({1\over1+s}\,B^{1/2}WB^{1/2}-B^{1/2}h_s(W)B^{1/2}\biggr)\,d\mu(s) \\
&\quad=B^{1/2}\biggl[\alpha I+\beta W+\gamma W^2
+\int_{(0,\infty)}\phi_s(W)\,d\mu(s)\biggr]B^{1/2}=B^{1/2}f(W)B^{1/2}.
\end{align*}
\end{proof}

When $f(0^+)<\infty$, we extend $f$ to $[0,\infty)$ continuously by $f(0):=f(0^+)$. Then,
when $A\ge0$ and $B>0$, $P_f(A,B)$ is well defined directly by \eqref{F-2.5} and it is
equal to the expression in \eqref{F-8.3}. (This extended definition has already
been used in the proof of Lemma \ref{L-3.18}.) With this definition of $P_f(A,B)$ for
$A\ge0$ and $B>0$ we furthermore have the following:

\begin{theorem}\label{T-6.3}
Assume that $f$ is an operator convex function on $(0,\infty)$ with $f(0^+)<\infty$. Let
$A,B\ge0$ with $A\le cB$ for some $c>0$. Then for any sequence $L_n\in B(\cH)^{++}$ such that
$\|L_n\|_\infty\to0$,
\begin{align}\label{F-8.13}
\lim_{n\to\infty}P_f(A,B+L_n)=\lim_{\eps\searrow0}P_f(A,B+\eps I)
=B^{1/2}f(D(A/B))B^{1/2}.
\end{align}
\end{theorem}

\begin{proof}
Set $W:=D(A/B)$; so $\|W\|_\infty\le c$.
For any $\delta>0$ define $f_\delta(t):=f(t+\delta)$ for $t>0$, which is operator convex on
$(-\delta,\infty)$. Note that
\begin{align*}
P_f(A,B+L_n)
&=(B+L_n)^{1/2}f\bigl((B+L_n)^{-1/2}A(B+L_n)^{-1/2}\bigr)(B+L_n)^{1/2}, \\
P_{f_\delta}(A,B+L_n)
&=(B+L_n)^{1/2}f\bigl((B+L_n)^{-1/2}A(B+L_n)^{-1/2}+\delta I\bigr)(B+L_n)^{1/2}.
\end{align*}
Since
\begin{align*}
\|(B+L_n)^{-1/2}A(B+L_n)^{-1/2}\|_\infty
&=  \|(B+L_n)^{-1/2}B^{1/2}WB^{1/2}(B+L_n)^{-1/2}\|_\infty \\
&\le c\|(B+L_n)^{-1/2}B(B+L_n)^{-1/2}\|_\infty \\
&=c\|B^{1/2}(B+L_n)^{-1}B^{1/2}\|_\infty\le c,
\end{align*}
one can estimate
\begin{align*}
&\sup_{n\ge1}\|P_{f_\delta}(A,B+L_n)-P_f(A,B+L_n)\|_\infty \\
&\qquad\le\sup_{n\ge1}\|B+L_n\|_\infty
\big\|f\bigl((B+L_n)^{-1/2}A(B+L_n)^{-1/2}+\delta I\bigr) \\
&\hskip5cm-f\bigl((B+L_n)^{-1/2}A(B+L_n)^{-1/2}\bigr)\big\|_\infty \\
&\qquad\le\Bigl(\|B\|_\infty+\sup_n\|L_n\|_\infty\Bigr)\sup_{t\in[0,c]}|f(t+\delta)-f(t)|
\ \longrightarrow\ 0\quad\mbox{as $\delta\searrow0$}.
\end{align*}
For every $\xi\in\cH$ with $\|\xi\|=1$, it follows that
\begin{align*}
\|(P_f(A,B+L_n)-B^{1/2}f(W)B^{1/2})\xi\|
&\le\|P_f(A,B+L_n)-P_{f_\delta}(A,B+L_n)\|_\infty \\
&\quad+\|(P_{f_\delta}(A,B+L_n)-B^{1/2}f_\delta(W)B^{1/2})\xi\| \\
&\quad+\|B^{1/2}f_\delta(W)B^{1/2}-B^{1/2}f(W)B^{1/2}\|_\infty,
\end{align*}
and the first and the third terms of the above right-hand side are arbitrarily small
independently of $n$ when $\delta>0$ is sufficiently small. Hence it suffices to show the
result for $f_\delta$ instead of $f$. So, replacing $f$ with $f_\delta$, we may and do assume
that $f'(0^+):=\lim_{t\to0^+}f'(t)>-\infty$. Now, define $f_0(t):=f(t)-\alpha-\beta t$ for
$t>0$, where $\alpha:=f(0^+)$ and $\beta:=f'(0^+)$. Then $f_0\in\OC_+$ with $f_0(0^+)=0$.
Since
\begin{align*}
P_f(A,B+L_n)&=P_{f_0}(A,B+L_n)+\alpha(B+L_n)+\beta A, \\
B^{1/2}f(W)B^{1/2}&=B^{1/2}f_0(W)B^{1/2}+\alpha B+\beta A,
\end{align*}
it suffices to show the result for $f_0$ instead of $f$. So we may finally assume that
$f\in\OC_+$ with $f(0^+)=0$. In this situation, note that if $0<B_1\le B_2$, then
$P_f(A,B_1)\ge P_f(A,B_2)$. Indeed, by Theorem \ref{T-6.2} and  Proposition
\ref{P-2.2}\,(vii) we have
$$
P_f(A,B_1)=\lim_{\eps\searrow0}P_f(A+\eps I,B_1+\eps I)
\ge\lim_{\eps\searrow0}P_f(A+\eps I,B_2+\eps I)=P_f(A,B_2).
$$
Therefore, since $L_n>0$ and $\|L_n\|_\infty\to0$, we easily see that both limits
$$
\lim_{\eps\searrow0}P_f(A,B+\eps I)\quad \mbox{and} \quad \lim_{n\to\infty}P_f(A,B+L_n)
\qquad\mbox{(SOT)}
$$
exist and are the same. Hence it remains to prove that
$\lim_{\eps\searrow0}P_f(A,B+\eps I)=B^{1/2}f(W)B^{1/2}$. The proof of this is similar
to (in fact, a bit easier than) that of Theorem \ref{T-6.2} by repeating the proof with
$A$, $B+\eps I$ in place of $A+\eps I$, $B+\eps I$. The details may be omitted here.
\end{proof}

In view of \eqref{f(0^+)-f'(infty)} and \eqref{F-2.5}, Theorems \ref{T-6.2} and
\ref{T-6.3} are rephrased as follows:

\begin{corollary}\label{C-6.4}
The following conditions are equivalent:
\begin{itemize}
\item[$(${\rm i}$)$] the limit in \eqref{F-8.1} exists for every $A,B\in B(\cH)^+$ such that
$cA\ge B$ for some $c>0$;
\item[$(${\rm ii}$)$] $f'(\infty)<\infty$.
\end{itemize}

In this case, for every $A,B$ as in $(${\rm i}$)$,
\begin{align}\label{F-8.14}
\lim_{\eps\searrow0}P_f(A+\eps I,B+\eps I)
=\lim_{\eps\searrow0}P_f(A+L_n,B)
=A^{1/2}\widetilde f(D(B/A))A^{1/2}\quad\mbox{$(${\rm SOT}$)$},
\end{align}
where $\widetilde f$ extends to $[0,\infty)$ by $\widetilde f(0)=\widetilde f(0^+)$ and
$L_n>0$, $\|L_n\|_\infty\to0$.
\end{corollary}

For simplicity of notations we set
\begin{align*}
(B(\cH)^+\times B(\cH)^+)_\le&:=\{(A,B)\in B(\cH)^+\times B(\cH)^+:
A\le cB\ \mbox{for some $c>0$}\}, \\
(B(\cH)^+\times B(\cH)^+)_\ge&:=\{(A,B)\in B(\cH)^+\times B(\cH)^+:
cA\ge B\ \mbox{for some $c>0$}\}.
\end{align*}
When $f(0^+)<\infty$ (resp., $f'(\infty)<\infty$), we extend $P_f$ to
$(B(\cH)^+\times B(\cH)^+)_\le$ (resp., $(B(\cH)^+\times B(\cH)^+)_\ge$) by defining
$P_f(A,B)$ by the expression in \eqref{F-8.3} or \eqref{F-8.13} (reps., \eqref{F-8.14}).

The joint operator convexity of $P_f$ in \cite[Theorem 2.2]{ENG} is extended as
follows, by a simple argument taking limits from Theorem \ref{T-6.2} or Corollary \ref{C-6.4}. 

\begin{proposition}\label{P-6.5}
If $f(0^+)<\infty$, then $(A,B)\mapsto P_f(A,B)$ is jointly operator convex on
$(B(\cH)^+\times B(\cH)^+)_\le$. If $f'(\infty)<\infty$, then $(A,B)\mapsto P_f(A,B)$ is
jointly operator convex on $(B(\cH)^+\times B(\cH)^+)_\ge$.
\end{proposition}

Thanks to the homogeneity $P_f(\alpha A,\alpha B)=\alpha P_f(A,B)$ for $\alpha>0$, the joint
operator convexity of $P_f(A,B)$ on $(B(\cH)^+\times B(\cH)^+)_\le$ (or
$(B(\cH)^+\times B(\cH)^+)_\ge$) is equivalent to the super-additivity, i.e.,
$$
P_f(A+C,B+D)\le P_f(A,B)+P_f(C,D)
$$
for $(A,B),(C,D)\in (B(\cH)^+\times B(\cH)^+)_\le$ (or $(B(\cH)^+\times B(\cH)^+)_\ge$).

Similarly, the properties in (vii) and (viii) of Proposition \ref{P-2.2} are extended as
follows:

\begin{proposition}\label{P-6.6}
Assume that $f\in\OC_+$ with $f(0^+)=0$. Then $P_f(A,B_1)\ge P_f(A,B_2)$ if
$(A,B_1)\in(B(\cH)^+\times B(\cH)^+)_\le$ and $B_1\le B_2$. Also,
$P_{\widetilde f}(A_1,B)\ge P_{\widetilde f}(A_2,B)$ if
$(A_1,B)\in(B(\cH)^+\times B(\cH)^+)_\ge$ and $A_1\le A_2$.
\end{proposition}

Another important property of $P_f$ is the monotonicity under positive linear maps,
summarized as follows:

\begin{theorem}\label{T-6.7}
Let $f$ be an operator convex function on $(0,\infty)$ and $\Phi:B(\cH)\to B(\cK)$ be a positive
linear map, where $\cK$ is another Hilbert space.
\begin{itemize}
\item[\rm(1)] If $\Phi(I)$ is invertible, then
\begin{align}\label{F-8.15}
\Phi(P_f(A,B))\ge P_f(\Phi(A),\Phi(B))
\end{align}
for all $A,B>0$.
\item[\rm(2)] If $f(0^+)<\infty$ and $\Phi(I)$ is not necessarily invertible, then
\eqref{F-8.15} holds for all $(A,B)\in(B(\cH)^+\times B(\cH)^+)_\le$.
\end{itemize}
\end{theorem}

\begin{proof}
(1)\enspace
Let $A,B>0$. Since $\Phi(B)$ is invertible, one can define a unital positive linear map
$$
\Psi(X):=\Phi(B)^{-1/2}\Phi(B^{1/2}XB^{1/2})\Phi(B)^{-1/2},\qquad X\in B(\cH).
$$
Then
\begin{align*}
\Phi(P_f(A,B))&=\Phi(B)^{1/2}\Psi(f(B^{-1/2}AB^{-1/2}))\Phi(B)^{1/2} \\
&\ge\Phi(B)^{1/2}f(\Psi(B^{-1/2}AB^{-1/2}))\Phi(B)^{1/2}=P_f(\Phi(A),\Phi(B)),
\end{align*}
where the inequality above is the Jensen operator inequality due to \cite[Theorem 2.1]{Ch}
and \cite{Da}.

(2)\enspace
By an approximation argument with $f_\delta(t):=f(t+\delta)$ as in the proof of Theorem
\ref{T-6.2}, we may assume that $f'(0^+)>-\infty$. Then define $f_0(t):=f(t)-\alpha-\beta t$
with $\alpha:=f(0^+)$ and $\beta:=f'(0^+)$, so $f_0\in\OC_+$ and $f_0(0^+)=0$. Since
\begin{align*}
\Phi(P_f(A,B))&=\Phi(P_{f_0}(A,B))+\alpha\Phi(B)+\beta\Phi(A), \\
P_f(\Phi(A),\Phi(B))&=P_{f_0}(\Phi(A),\Phi(B))+\alpha\Phi(B)+\beta\Phi(A),
\end{align*}
we may and do assume that $f\in\OC_+$ with $f(0^+)=0$.

Take a state $\omega(X):=\<\xi,X\xi\>$ on $B(\cH)$ where $\xi$ is any unit vector in $\cH$.
For any $n\in\bN$ set $\Phi_n(X):=\Phi(X)+n^{-1}\omega(X)I$ for $X\in B(\cH)$. For any
$(A,B)\in(B(\cH)^+\times B(\cH)^+)_\le$ and $\eps>0$, as in the proof of (1) above (with
$A\ge0$ in the present case), one can see that
$$
\Phi_n(P_f(A,B+\eps I))\ge P_f(\Phi_n(A),\Phi_n(B+\eps I)).
$$
Since $P_f(A,B+\eps I)\le P_f(A,B)$ by Proposition \ref{P-6.6},
$$
\Phi_n(P_f(A,B))\ge P_f(\Phi_n(A),\Phi_n(B+\eps I)).
$$
Now, for every $\delta>0$ one can choose an $n_0\in\bN$ and an $\eps>0$ such that
$$
\Phi_n(B+\eps I)=\Phi(B)+\eps\Phi(I)+n^{-1}\omega(B+\eps I)I\le\Phi(B)+\delta I,
\qquad n\ge n_0.
$$
Hence by Proposition \ref{P-6.6} again,
$$
\Phi_n(P_f(A,B))\ge P_f(\Phi_n(A),\Phi(B)+\delta I),\qquad n\ge n_0.
$$
Letting $n\to\infty$ implies that $\Phi(P_f(A,B))\ge P_f(\Phi(A),\Phi(B)+\delta I)$. Finally,
letting $\delta\searrow0$ gives the result due to Theorem \ref{T-6.3}.
\end{proof}

As a special case of Theorem \ref{T-6.7} we obtain the \emph{transformer inequality} of $P_f$,
opposite to that of operator connections \cite{KA}, as follows: If $f(0^+)<\infty$ and
$(A,B)\in(B(\cH)^+\times B(\cH)^+)_\le$, then for any $T\in B(\cH)$,
$$
T^*P_f(A,B)T\ge P_f(T^*AT,T^*BT),
$$
and equality holds in the above if $T$ is invertible.

\begin{proposition}\label{P-6.8}
\begin{itemize}
\item[\rm(1)] Assume that $f'(\infty)=\infty$. If $A,B\in B(\cH)^+$ and $s(A)\not\le s(B)$,
then the limit in \eqref{F-8.1} does not exist.
\item[\rm(2)] Assume that $f(0^+)=\infty$. If $A,B\in B(\cH)$ and $s(B)\not\le s(A)$, then
the limit in \eqref{F-8.1} does not exist.
\end{itemize}
\end{proposition}

\begin{proof}
(1)\enspace
Let $A,B\in B(\cH)^+$ and assume that $s(A)\not\le s(B)$. Then there is a unit vector
$\xi\in\cH$ such that $A\xi\ne0$ but $B\xi=0$. Consider a state $\omega(X):=\<\xi,X\xi\>$ on
$B(\cH)$. Note that $\alpha:=\omega(A)>0$ and $\omega(B)=0$.
From the monotonicity property of $P_f$ in Theorem \ref{T-6.7}\,(1), for any $\eps>0$
we have
\begin{align}\label{F-14}
P_f(\omega(A+\eps I),\omega(B+\eps I))\le\omega(P_f(A+\eps I,B+\eps I)).
\end{align}
Now, assume that $f'(\infty)=\infty$. The left-hand side of \eqref{F-14} is
$P_f(\alpha+\eps,\eps)$ and
$$
\lim_{\eps\searrow0}P_f(\alpha+\eps,\eps)
=\lim_{\eps\searrow0}\eps f\biggl({\alpha+\eps\over\eps}\biggr)
=\lim_{\eps\searrow0}(\alpha+\eps)
{\eps\over\alpha+\eps}\,f\biggl({\alpha+\eps\over\eps}\biggr)=\infty.
$$
Hence the right-hand side of \eqref{F-14} diverges, so
$\lim_{\eps\searrow0}P_f(A+\eps I,B+\eps I)$ does not exist.

(2) is immediate form (1) in view of \eqref{f(0^+)-f'(infty)} and \eqref{F-2.5}.
\end{proof}

\begin{Remark}\label{R-6.9}\rm
When $f(0^+)<\infty$ and $s(A)\le s(B)$ (or when $f'(\infty)<\infty$ and $s(B)\le s(A)$),
both cases where the limit in \eqref{F-8.1} does or does not exist can occur. For example,
let $\{e_n\}_{n=1}^\infty$ be an orthonormal basis of $\cH$, and let $A=\sum_na_nE_n$ and
$B=\sum_nb_nE_n$, where $a_n,b_n>0$ are bounded and $E_n$ is the rank one projection onto
$\bC e_n$. Then $s(A)=s(B)=I$, and
$$
\lim_{\eps\searrow0}P_f(A+\eps I,B+\eps I)
=\lim_{\eps\searrow0}\sum_n(b_n+\eps)f\biggl({a_n+\eps\over b_n+\eps}\biggr)E_n
\quad (SOT)
$$
exists if and only if $\sup_nb_nf(a_n/b_n)<\infty$. When $f(t)=t^2$, the limit exist if
$a_n=1/n$ and $b_n=1/n^2$, but the limit does not exists if $a_n=1/n$ and $b_n=1/n^3$.
\end{Remark}

We extend the pmi part of AH type inequalities in Corollary \ref{C-3.8} to non-invertible
operators with $A\le cB$ or $cA\ge B$.

\begin{proposition}\label{P-6.10}
If $f\in\OC_+^1$ with $f(0^+)=0$ is pmi, then $P_f$ satisfies \eqref{ando-hiai} for every
$(A,B)\in(B(\cH)^+\times B(\cH)^+)_\le$ and all $p\in(0,1]$. If $g\in\OMD_+^1$ is pmi,
then $P_g$ satisfies \eqref{ando-hiai} for every $(A,B)\in(B(\cH)^+\times B(\cH)^+)_\ge$
and all $p\in(0,1]$.
\end{proposition}

\begin{proof}
We may prove the result for $P_f$ only. Assume that $P_f(A,B)\le I$. For any $\eps>0$, by
Proposition \ref{P-6.5} one has
$$
P_f\biggl({A+\eps I\over1+\eps},{B+\eps I\over1+\eps}\biggr)
\le{P_f(A,B)+\eps I\over1+\eps}\le I.
$$
Hence Corollary \ref{C-3.8}\,(1) for $P_f$ implies that
$$
P_f\biggl(\biggl({A+\eps I\over1+\eps}\biggr)^p,\biggl({A+\eps I\over1+\eps}\biggr)^p
\biggr)\le I,
$$
so that $P_f((A+\eps I)^p,(B+\eps I)^p)\le(1+\eps)^pI$. For any $\delta>0$, since
$(B+\eps I)^p\le B^p+\delta I$ for $\eps>0$ sufficiently small, it follows from
Proposition \ref{P-2.2}\,(vii) that
$$
P_f((A+\eps I)^p,B^p+\delta I)\le (1+\eps)^pI
$$
for all $\eps>0$ sufficiently small. Letting $\eps\searrow0$ with $\delta$ fixed we obtain
$P_f(A^p,B^p+\delta I)\le I$ for any $\delta>0$. Hence, letting $\delta\searrow0$ gives the
result by Theorem \ref{T-6.3}.
\end{proof}

In the rest of the section we assume that $\cH$ is finite-dimensional. Then for
$A,B\ge0$, note that $A\le cB$ for some $c>0$ if and only if $s(A)\le s(B)$. When $f$ is
any continuous function on $[0,\infty)$, it is not difficult to see that for any $A,B\ge0$
with $s(A)\le s(B)$ and for any $p>0$,
\begin{align}
\lim_{\eps\searrow0}P_f((A+\eps I)^p,(B+\eps I)^p)
&=\lim_{\eps\searrow0}P_f((A+\eps s(A)^\perp)^p,(B+\eps s(B)^\perp)^p) \nonumber\\
&=B^{p/2}f(D(A^p/B^p))B^{p/2}\quad\mbox{(in $\|\cdot\|_\infty$)}.\label{P_f-conv}
\end{align}
Indeed, $P_f((A+\eps I)^p,(B+\eps I)^p)$ is the direct sum of
$P_f((A+\eps s(B))^p,(B+\eps s(B))^p)$ on $s(B)\cH$ and $\eps^pf(1)s(B)^\perp$, and the first
component converges to $B^{p/2}f(D(A^p/B^p))B^{p/2}$ on $s(B)\cH$ as $\eps\searrow0$. The
proof of the second limit formula in \eqref{P_f-conv} is similar. So it is easy to extend
some AH type inequalities in Section 3 to positive semidefinite matrices. For example, we
have the following:

\begin{proposition}\label{P-6.11}
If $f\in\OC_+^1$ with $f(0^+)=0$ is pmi and $A,B$ are positive semidefinite matrices with
$s(A)\le s(B)$, then
\begin{align*}
&P_f(A^p,B^p)\prec_{w\log}\|P_f(A,B)\|_\infty^{1-p}P_f(A,B)^{2p-1},
\qquad1/2\le p<1, \\
&\|P_f(A,B)\|_\infty^{1-p}P_g(A,B)^{2p-1}\prec_{w\log}P_f(A^p,B^p),
\qquad1\le p\le2.
\end{align*}
\end{proposition} 

\begin{proposition}\label{P-6.12}
If $h\in OM_+^1$ is pmi and $A,B$ are positive semidefinite matrices with $s(A)\le s(B)$,
then
$$
P_{t^nh}(A,B)\le I\ \implies\ P_{t^nh}(A^p,B^p)\le I
$$
holds for all $p\in (0,1/2]$ and $n\ge 2$.
\end{proposition} 

Furthermore, Proposition \ref{P-3.9} for $P_f$ can be extended to positive semidefinite
matrices under an assumption on $f$.

\begin{proposition}\label{P-6.13}
Let $f\in\OC_+^1$ with $f(0^+)=0$ and assume that $\lim_{t\to0^+}f(t^p)/f(t)^p$ exists for
all $p\in(0,1)$. Let $A,B$ be positive semidefinite matrices $A,B$ with $s(A)\le s(B)$.
Then for every $p\in[1/2,1]$, 
$$
P_f(A,B)\le I\ \implies\ P_f(A^p,B^p)\le
\bigg\|{f(D(A/B)^p)\over f(D(A/B))^p}\,s(B)+s(B)^\perp\bigg\|_\infty I,
$$
where ${f(D(A/B)^p)\over f(D(A/B))^p}s(B)$ is defined as the functional calculus of
$D(A/B)s(B)$ by the function $f(t^p)/f(t)^p$ on $[0,\infty)$ whose value at $t=0$ is
$\lim_{t\to0^+}f(t^p)/f(t)^p$.
\end{proposition}

\begin{proof}
For each $\eps>0$, since $P_f\bigl({A+\eps I\over1+\eps},{B+\eps I\over1+\eps}\bigr)\le I$,
Proposition \ref{P-3.9} implies that
$$
P_f((A+\eps I)^p,(B+\eps I)^p)\le(1+\eps)^p
\bigg\|{f(C_\eps^p)\over f(C_\eps)^p}\bigg\|_\infty I,
$$
where $C_\eps:=(B+\eps I)^{-1/2}(A+\eps I)(B+\eps I)^{-1/2}$. Note that
\begin{align*}
C_\eps&=(B+\eps s(B))^{-1/2}(A+\eps s(B))(B+\eps s(B))^{-1/2}+s(B)^\perp \\
&\longrightarrow D(A/B)+s(B)^\perp\quad\mbox{as $\eps\searrow0$}.
\end{align*}
Hence, under the assumption that $f(t^p)/f(t)^p$ is continued at $t=0$ as stated, we have
$$
{f(C_\eps^p)\over f(C_\eps)^p}\ \longrightarrow
\ {f(D(A/B)^p)\over f(D(A/B))^p}\,s(B)+s(B)^\perp,
$$
which with \eqref{P_f-conv} implies the assertion.
\end{proof}

\begin{Remark}\label{R-6.14}\rm
Since $f(t)=th(t)$ with $h\in\OM_+^1$ by Proposition \ref{P-2.1}, the assumption on $f$ in
Proposition \ref{P-6.13} is equivalent to that $\lim_{t\to0^+}h(t^p)/h(t)^p$ exists for
all $p\in(0,1)$. From Lemma \ref{L-3.16}, this condition holds  if $h$ is geometrically convex.
But it is not always the case. For any $h\in\OM_+^1$ let $h^\perp(t):=t/h(t)$,
the \emph{dual} function of $h$ (\cite{KA}). Then
$$
\lim_{t\to0^+}h^\perp (t^p)/h^\perp (t)^p =
\Bigl(\lim_{t\to0^+}h(t^p)/h(t)^p\Bigr)^{-1}
$$
as long as the limit in the right-hand side exists in $[0,\infty]$. When $h(t):=(t-1)/\log t$,
the representing function of the \emph{logarithmic mean},
$\lim_{t\to0^+}h(t^p)/h(t)^p=0$, so $\lim_{t\to0^+}h^\perp (t^p)/h^\perp (t)^p=\infty$.
\end{Remark}

Finally, we extend some operator norm inequalities in Section 5.2 to positive semidefinite
matrices. For positive semidefinite matrices $A,B$ and $\alpha,\beta\in\bR\setminus\{0\}$, we
define
$$
\exp(\alpha\log A\,\dot+\,\beta\log B)
:=P_0\exp(\alpha P_0(\log A)P_0+\beta P_0(\log B)P_0),
$$
where $P_0:=s(A)\wedge s(B)$, the orthogonal projection onto the intersection of the supports
of $A,B$. The next lemma is useful.

\begin{lemma}\label{L-6.15}
Let $A,B$ be positive semidefinite matrices. Assume either that $\alpha,\beta>0$, or that
$s(A)\le s(B)$, $\beta<0$ and $\alpha+\beta>0$. Then
\begin{align*}
\exp(\alpha\log A\,\dot+\,\beta\log B)
&=\lim_{\eps\searrow0}\exp\{\alpha\log(A+\eps s(A)^\perp)
+\beta\log(B+\eps s(B)^\perp)\}.
\end{align*}
\end{lemma}

\begin{proof}
When $\alpha,\beta>0$, the asserted formula was shown in \cite[Lemma 4.1]{HiPe}. Now,
assume that $s(A)\le s(B)$, $\beta<0$ and $\alpha+\beta>0$. Note that
\begin{align}
&\exp\bigl\{\alpha\log(A+\eps s(A)^\perp)+\beta\log(B+\eps s(B)^\perp)\bigr\} \nonumber\\
&\quad=s(B)\exp\bigl\{\alpha s(B)\log(A+\eps(s(B)-s(A))
+(-\beta)s(B)\log(B^{-1}s(B)))\bigr\} \nonumber\\
&\qquad\quad+\eps^{\alpha+\beta}s(B)^\perp. \label{F-6.18}
\end{align}
From the first case applied to $A$ and $B^{-1}$ restricted to the range of $s(B)$,
the right-hand side of \eqref{F-6.18} converges as $\eps\searrow0$ to
$$
s(A)\exp\bigl\{\alpha s(A)\log A+(-\beta)s(A)(\log(B^{-1}s(B)))s(A)\bigr\}
=\exp(\alpha\log A\,\dot+\,\beta\log B).
$$
\end{proof}

The next proposition extends Corollary \ref{C-5.3} to the non-invertible case, though
restricted to matrices. (Related results for infinite-dimensional operators are found in
\cite[Section 4]{Hi1}.)

\begin{proposition}\label{P-6.16}
Assume that $h\in\OM_+^1$ is pmi and $\alpha:=h'(1)\in(0,1)$ (equivalently, $h\ne1,t$). Let
$A,B$ be positive semidefinite matrices and $p>0$. Then
\begin{align}\label{F-6.19}
\|P_{h^*}(B^p,A^p)^{1/p}\|_\infty
\le\|\exp(\alpha\log A\,\dot+\,(1-\alpha)\log B)\|_\infty.
\end{align}
Moreover, if $s(A)\le s(B)$, then any $n\in\bN$,
\begin{align}\label{F-6.20}
\|\exp((n+\alpha)\log A\,\dot+\,(1-n-\alpha)\log B)\|_\infty
\le\|P_{t^nh}(A^p,B^p)^{1/p}\|_\infty,
\end{align}
where $P_{t^nh}(A^p,B^p)$ is defined by the limit in \eqref{P_f-conv}.
\end{proposition}

\begin{proof}
It follows from \eqref{Log-Euclid1} that
\begin{align*}
&\|P_{h^*}((A+\eps s(A)^\perp)^p,(B+\eps s(B)^\perp)^p)^{1/p}\|_\infty \\
&\qquad\le\|\exp(\alpha\log(A+\eps s(A)^\perp)+(1-\alpha)\log(B+\eps s(B)^\perp))\|_\infty.
\end{align*}
Hence letting $\eps\searrow0$ gives \eqref{F-6.19} by \eqref{P_f-conv} and Lemma \ref{L-6.15}.
When $s(A)\le s(B)$, \eqref{F-6.20} follows similarly from \eqref{Log-Euclid2},
\eqref{P_f-conv} and Lemma \ref{L-6.15}.
\end{proof}

In particular, for power functions $t^\alpha$ we state the following:

\begin{corollary}\label{C-6.17}
Let $A,B$ be positive semidefinite matrices and $p>0$. For any $\alpha\in(0,1)$,
$$
(B^p\#_\alpha A^p)^{1/p}\prec_{\log}\exp(\alpha\log A\,\dot+\,(1-\alpha)\log B).
$$
If $s(A)\le s(B)$, then for any $\alpha>1$,
$$
\exp(\alpha\log A\,\dot+\,(1-\alpha)\log B)\prec_{\log}
P_{t^\alpha}(A^p,B^p)^{1/p}.
$$
\end{corollary}

\subsection*{Acknowledgements}

The work of F.~Hiai and Y.~Seo was supported in part by JSPS KAKENHI Grant Numbers
JP17K05266 and JP19K03542, respectively.

\end{document}